\documentclass[a4paper, 10pt]{amsart}
\usepackage{amsthm}
\usepackage[]{amsmath}
\usepackage{amssymb}
\usepackage{enumerate}
\usepackage{tabularx}
\usepackage{supertabular}
\usepackage[]{color}
\usepackage[left=3.0cm, right=3.0cm, top=3.5cm, bottom=3.5cm]{geometry}
\usepackage[colorlinks]{hyperref}
\usepackage{tikz}
\usepackage{multirow}
\usepackage{diagbox}
\usepackage{subcaption}
\usepackage[ruled, vlined]{algorithm2e}
\usepackage{xcolor}
\usepackage[utf8]{inputenc}
\usepackage{booktabs}
\usepackage{xltabular}
\usepackage[]{longtable}
\usepackage{float}

\usetikzlibrary{calc}

\allowdisplaybreaks[4]

\newcommand\justsubstitute[2]{#2}
\newcommand\revise[2]{#2}

\title[Particle \justsubstitute{Tracking}{Locating} on Patch
Searching]{An efficient particle \revise{tracking}{locating} method on
unstructured meshes in two and three dimensions based on patch
searching}

\author[S. Chen]{Shuang Chen} \address{Yau Mathematical Sciences
Center, Tsinghua University, Beijing 100084, P.R. China}
\email{chenshua24@mails.tsinghua.edu.cn}

\author[F. Yang]{Fanyi Yang} \address{School of Mathematics
, Sichuan University, Chengdu 610065, P.R. China}
\email{yangfanyi@scu.edu.cn}

\newcommand{\bm}[1]{\boldsymbol{#1}}
\newcommand{\mb}[1]{\mathbb{#1}}
\newcommand{\mc}[1]{\mathcal{#1}}
\newcommand{\mr}[1]{\mathrm{#1}}
\newcommand{\bmr}[1]{\bm{\mr{#1}}}
\newcommand{\wt}[1]{\tilde{#1}}
\newcommand{\wh}[1]{\hat{#1}}

\newcommand{\mD}[1]{\mc{D}(#1)}
\newcommand{\tred}[1]{\textcolor{red}{#1}}

\def\mB{\mc{B}}

\def\mT{\mc{T}}
\def\MTh{\mc{T}_h}
\def\MCs{\mc{C}_{s}}
\def\MCsc{\mc{C}_s^{\circ}}
\def\MCscI{\mc{C}_{s}^{\circ, \bmr{c}}}
\def\MCscB{\mc{C}_{s}^{\circ, \bmr{b}}}
\def\MCscE{\mc{C}_{s}^{\circ, \bmr{e}}}
\def\MCscF{\mc{C}_{s}^{\circ, \bmr{f}}}
\def\MFh{\mc{F}_h}
\def\MFhi{\mc{F}_h^i}
\def\MFhb{\mc{F}_h^b}
\def\MEh{\mc{E}_h}
\def\MEhi{\mc{E}_h^i}
\def\MEhb{\mc{E}_h^b}

\def\mN{\mc{N}}
\def\mE{\mc{E}}
\def\mF{\mc{F}}
\def\mM{\mc{M}}

\def\mC{\mc{C}}

\def\MNh{\mc{N}_h}
\def\MNhi{\mc{N}_h^i}
\def\MNhb{\mc{N}_h^b}

\def\MMs{\mc{M}_s}

\def\MCMEhc{\mc{C}_{\MEh}^{\circ}}
\def\MCec{\mc{C}_e^{\circ}}
\def\MCBvc{\mc{C}_{B_{\bv}}^{\circ}}

\def\MCNhsc{\mc{C}_{\mc{N}_h}^{\circ}}
\def\MCNhssc{\mc{C}_{\mc{N}_h^*}^{\circ}}

\def\bv{\bm{\nu}}
\def\bvs{\bm{\varsigma}}
\def\bom{\bm{\omega}}
\def\bc{\bm{c}}

\def\mO{\mc{O}}
\def\bchi{\bm{\chi}}

\def\sign{\mr{sign}}
\def\bzta{\bm{\zeta}}
\def\bq{\bm{q}}

\newtheorem{lemma}{Lemma}
\newtheorem{remark}{Remark}

\newcommand{\step}[1]{\ifmmode{\bmr{#1}}^*\else${\bmr{#1}}^*$\fi}

\begin{document}

\maketitle

\begin{abstract}
  We present a particle \revise{tracking}{locating} method for
  unstructured meshes in two and three dimensions. 
  Our algorithm is based on a patch searching process, and includes
  two steps. 
  We first locate the given point to a patch near a vertex, and then
  the host element is determined within the patch domain.
  Here, the patch near a vertex is the domain of elements around this
  vertex.
  We prove that in the first step the patch can be rapidly identified
  by constructing an auxiliary Cartesian grid with a prescribed
  resolution. 
  Then, the second step can be converted into a searching problem,
  which can be easily solved by searching algorithms.
  Only coordinates to particles are required in our method.
  We conduct a series of numerical tests in two and three dimensions
  to illustrate the robustness and efficiency of our method.

\noindent \textbf{keywords}: 
particle \revise{tracking}{locating}; 
unstructured mesh;
auxiliary Cartesian grid;

\end{abstract}

\section{Introduction}
\label{sec_introduction}
To locate the simulated particles in the mesh system is frequently 
employed in fields of computational fluid dynamics, such as Lagrangian
particle tracking methods, particle-in-cell methods and multi-phase
flow simulations \cite{Sani2009set, Chorda2002efficient}.
The particle \revise{tracking}{locating} problem on a mesh reads: {\it
given coordinates of a point, then determine the element that contains
it (the host element)}.
In the case of the uniform Cartesian mesh, this problem is
very simple and straightforward. 
But for unstructured meshes, which are widely used in problems
involving complex geometry and bring more flexibility, developing an
efficient \revise{tracking}{locating} method is not trivial.

For the particle \revise{tracking}{locating} problem on unstructured
meshes, several methods have been proposed, which can be roughly
classified into two types \cite{Lowier1990vectorized, Wang2022GPU,
Li2019fast}: the auxiliary structured Cartesian grid method and the
neighbourhood searching method. 
For the first method, the main idea is to construct a background
structured grid entirely covering the computational mesh, and
for each cell in the auxiliary grid, a list of unstructured elements
intersecting this cell is stored. 
The point is first located on the auxiliary structured grid, which is
a very easy task,
and then the point-in-element tests are conducted on elements in the
corresponding list to seek the host element. 
We refer to \cite{Seldner1988algorithms, Muradoglu2006auxiliary,
Li2019fast} for some typical methods.
The neighbourhood searching method is another popular and widely
used locating method. 
This method requires the trajectory of the given particle and the host
element of the starting point.
In performing this method, a next possible host element in
neighbours of that host element is determined by some specific
algorithms.
The searching process will be repeated until the final host element is
reached. 
Generally, a searching path along the trajectory will be constructed
for each particle and eventually leads to the host element.
In \cite{Lowier1995robust}, the proposed method used the barycentric
coordinates of the point relative to the triangle to find the next
search direction.
In \cite{Chen1999new}, the authors developed a directed search method
by examining the elements passed through by the given particle using
determinants.
In \cite{Martin2009particle}, the authors utilized the outward normal
vector on faces to find the next possible host element. 
A similar searching method by seeking the face that the particle
crosses it to leave the current element was given in
\cite{Macpherson2009particle},
and this method can be used for polygonal elements. 
In \cite{Wang2022GPU}, the authors presented an improved searching
approach and combined the CPU-acceleration technique.
More methods of such type are referred to \cite{Sani2009set,
Capodaglio2017particle, Haselbacher2007efficient, Li2001efficient,
Chorda2002efficient, Kuang2008new}.
The neighbourhood searching method is usually simple to implement, has
a great efficiency, and allows arbitrary convex polygonal (polyhedral)
elements in the mesh. 
But if the starting position for a particle is not available, the
computational efficiency of such methods will be greatly affected
because finding a suitable starting point is also not a trivial task
\cite{Li2019fast}.

\revise{In this paper, we propose a particle tracking method for
unstructured meshes based on a patch searching procedure in two and
three dimensions.}{In this paper, we propose a particle locating
method for triangular and tetrahedral meshes in two and three
dimensions, based on a patch searching procedure.}
Here, the patch near a vertex/an edge is the domain of the union of
elements around this vertex/this edge. 
Our algorithm mainly consists of two steps. 
For a given point, we first search a vertex whose corresponding patch
contains the point, and the host element is further determined in the
patch domain.
We show that the first step can be rapidly implemented by
constructing a background Cartesian grid. 
Similar to the auxiliary grid method, we overlay a structured
Cartesian grid with a prescribed resolution over the unstructured
mesh.
We prove that any structured cell will be entirely
contained in a patch, which allows us to introduce a mapping from
Cartesian cells to vertices. 
From this mapping, the vertex for the given point can be easily found. 
In two dimensions, the second step is then shown to be equivalent to
locating an angle in an ascending sequence, which can be readily
solved. 
In three dimensions, the second step needs an extra moving step, where
we construct a new point from the given point sharing the same host
element. 
The new point will be far from all vertices, which allows us to prove 
that this point will be contained in a patch near an edge.
Then the host element can be
determined by locating a vector in a plane as two dimensions. 
Compared to traditional auxiliary structured grid methods, any
point-in-element test is not required in our algorithm, which can
significantly save the computational cost and increase the efficiency.
Another advantage in our method is that only coordinates of the
given point is needed, and the locating time is robust to the
position of the point.
The grid spacing of the Cartesian grid is explicitly given, and there
is no need to construct a very fine grid in our method.
The robustness and the efficiency are numerically confirmed by a
series of tests in two and three dimensions. 
\revise{}{ 
Currently, the proposed method and the theoretical analysis
are established on triangular (tetrahedral) meshes.
We additionally provide a numerical test on a 2D polygonal mesh to
demonstrate the performance of the algorithm. 
The detailed extension to polygonal (polyhedral) meshes is considered
in a future work.}
In the preparation stage, the main step is to establish the relation
mappings from Cartesian cells to vertices and edges. 
The details of the computer implementation are presented, and the
computational cost of the initialization is demonstrated to increase
linearly with the number of elements.

The rest of this paper is organized as follows. 
In Section \ref{sec_preliminaries}, we introduce the notation related
to the partition. 
In Section \ref{sec_method}, we present the locating algorithm in both
two and three dimensions. 
The details to the computer
implementation are described in Section \ref{sec_implement}.
In Section \ref{sec_numericalresults}, we conduct a series of
numerical tests to demonstrate the numerical performance.
\revise{}{
Section \ref{sec_conclusion} concludes the paper.  
Finally, a list of
notation can be found at the end of the paper. 
}

\section{Preliminaries}
\label{sec_preliminaries}
We first introduce the notation related to the domain and
the partition.
Let $\Omega \subset \mb{R}^d(d = 2, 3)$ be a polygonal (polyhedral)
domain in two or three dimensions, and we let $\MTh$ be a
quasi-uniform partition over $\Omega$ into a family of triangles
(tetrahedrons). 
For any $K \in \MTh$, we denote by $h_K$ the diameter of the
circumscribed ball of $K$, and by $\rho_K$ the radius of the
inscribed ball of $K$, and by $w_K$ the width of $K$.
The definitions indicate that $2\rho_K \leq w_K \leq h_K$ for $\forall
K \in \MTh$.
We define 
\begin{displaymath}
  h :=  \max_{K \in \MTh} h_K, \quad w :=  \min_{K \in \MTh} w_K,
  \quad \rho := \min_{K \in \MTh} \rho_K, 
\end{displaymath}
where $h$ is the mesh size to $\MTh$.
The mesh $\MTh$ is assumed to be quasi-uniform in the sense that there
exists a constant $C_{\nu}$ independent of $h$ such that $h \leq
C_{\nu} \rho$.
The quasi-uniformity of the mesh brings a minimum angle
condition to $\MTh$: there exists a constant $\alpha$ depending on
$C_{\nu}$ such that \cite{Brandts2008equivalence}
\begin{equation}
  \begin{aligned}
    \alpha_K & \geq \alpha, \quad \forall K \in \MTh, \quad d = 2,
    \\ 
    \alpha_{D, K}  \geq \alpha, \quad \alpha_{F, K} &\geq \alpha,
    \quad \forall K \in \MTh, \quad d = 3, 
  \end{aligned}
  \label{eq_miniangle}
\end{equation}
where $\alpha_K$ is the minimum angle of the triangle $K$ in two
dimensions, and in three dimensions, $\alpha_{D, K}$ is the minimum of
values of dihedral angles between faces of $K$ and $\alpha_{F, K}$ is
the minimum angle in all four triangular faces of $K$.

We denote by $\MNh$ the set of all nodes in $\MTh$, and we 
decompose $\MNh$ into $\MNh = \MNhi + \MNhb$, where $\MNhi := \{\bv
\in \MNh: \ \bv \in \Omega\}$ and  $\MNhb := \{ \bv \in \MNh: \ \bv
\in \partial \Omega\}$ consist of all interior nodes and the nodes
lying on the boundary $\partial \Omega$, respectively, see
Fig.~\ref{fig_node}.
For any $\bv \in \MNh$, we denote by $\mT_{\bv} := \{K \in \MTh: \ \bv
\in \partial {K} \}$ the set of elements sharing a common vertex
$\bv$. 
Let $\mD{\mT_{\bv}} := \text{Int}(\bigcup_{K \in \mT_{\bv}}
\overline{K})$ be the open domain corresponding to
$\mT_{\bv}$, see the right figure in Fig.~\ref{fig_node}.
It is noted that the patch $\mT_{\bv}$ and its domain 
$\mD{\mT_{\bv}}$ play an important role in our
\revise{tracking}{locating} algorithm.
Let $B(\bm{z}, r)$ be the disk (ball) centered at the position
$\bm{z}$ with the radius $r$, and we let $\partial B(\bm{z}, r)$ be
the sphere of $B(\bm{z}, r)$.
For any interior node $\bv \in \MNhi$, there holds
$B(\bv, w^*) \subset \mD{\mT_{\bv}}$ for any $w^* < w$, while for any
node $\bv \in \MNhb$, we have that $(B(\bm{v}, w^*) \cap
\Omega) \subset \mD{\mT_{\bv}}$.
For any $K \in \MTh$, we let $\mN_K := \{ \bm{w} \in \MNh: \ \bm{w}
\in \partial K \}$ be the set of all vertices of $K$.

\begin{figure}[htp]
  \centering
  \begin{minipage}[t]{0.43\textwidth}
    \begin{center}
      \begin{tikzpicture}[scale=1]
        \draw[thick, dashed] (0, 2) -- (0, 0) -- (3, 0);
        \draw[thick] (0, 0) -- (1, 0) -- (0, 1) -- (0, 0);
        \draw[thick] (0, 1) -- (0.85, 1.15) -- (1, 0);
        \draw[thick] (0, 0) -- (0.85, 1.15) -- (1.8, 0.7) -- (1, 0);
        \draw[fill=red, red] (0, 0) circle [radius=0.06];
        \draw[fill=red, red] (1, 0) circle [radius=0.06];
        \draw[fill=red, red] (0, 1) circle [radius=0.06];
        \draw[fill=blue, blue] (0.42, 0.57) circle [radius=0.06];
        \draw[fill=blue, blue] (0.85, 1.15) circle [radius=0.06];
        \draw[fill=blue, blue] (1.8, 0.7) circle   [radius=0.06];
      \end{tikzpicture}
    \end{center}
  \end{minipage}\begin{minipage}[t]{0.42\textwidth}
    \begin{center}
      \begin{tikzpicture}[scale=1.05]
        \foreach \x in {30, 90, 150, 210, 270, 330}
        \draw[thick, dashed] (0, 0) -- ({cos(\x)}, {sin(\x)});
        \foreach \x in {30, 90, 150, 210, 270, 330}
        \draw[thick] ({cos(\x)}, {sin(\x)}) -- ({cos(\x + 60)}, {sin(\x
        + 60)});
        \draw[blue, fill=blue] (0, 0) circle [radius=0.06];
        \node[below right] at (0, -0.1) {$\bv$};
      \end{tikzpicture}
    \end{center}
  \end{minipage}
  \caption{red: nodes in $\MNhb$, blue: nodes in $\MNhi$ (left)
  / the patch domain $\mD{\mT_{\bv}}$ (right)}
  \label{fig_node}
\end{figure}

Let $\MEh$ be the set of all edges of the partition $\MTh$, and also
$\MEh$ is decomposed into $\MEh = \MEhi + \MEhb$, where
$\MEhi$ and $\MEhb$ are sets of all interior and boundary edges,
respectively.
For any $e \in \MEh$, we let $\mN_e := \{ \bv \in \MNh: \
\bv \in \overline{e}\} $ be the set of its vertices, 
and let
$\mT_e := \{K \in \MTh: \ e
\subset \partial K \}$ be the set of elements sharing the common edge
$e$. 
Particularly in two dimensions, the set $\mT_e$ has two elements for
$e \in \MEhi$ while $\mT_e$ only has one element for $e \in \MEhb$.
For any $\mT_e$, we define the corresponding domain 
$\mD{\mT_e} := \text{Int}(\bigcup_{K \in \mT_e} \overline{K})$.
For any node $\bv \in \MNh$, we define $\mE_{\bv} := \{e \in \MEh: \
\bv \in \mN_{e}\}$ as the set of all edges sharing a common
vertex $\bv$.

In three dimensions, we define $\MFh$ as the set of all
two-dimensional faces in $\MTh$, and
$\MFh$ is still decomposed into $\MFh = \MFhi + \MFhb$, where $\MFhi$
and $\MFhb$ consist of all interior faces and the faces lying on the
boundary $\partial \Omega$, respectively. 
For any $f \in \MFh$,
we let $\mN_f := \{ \bv \in \MNh: \  \bv \in
\overline{f}\}$ be the set of vertices, and let $\mT_f := \{K \in
\MTh: \ f \subset \partial K \}$ be the collection of all elements
sharing the common face $f$. 
Similarly, $\mT_f$ has two/one elements for $f \in \MFhi$/$f \in
\MFhb$.

In our algorithm, we are required to construct an auxiliary regular
Cartesian grid that covers the whole computational
domain with a prescribed resolution. 
For this goal, we select a simple rectangular (cuboid) domain such
that $\Omega \subset \Omega^*$, where $\Omega^*$ can be described as
\begin{displaymath}
  \Omega^* = \begin{cases} (x_{\min}, x_{\max}) \times (y_{\min},
    y_{\max}), & d = 2, \\ 
    (x_{\min}, x_{\max}) \times (y_{\min}, y_{\max}) \times (z_{\min},
    z_{\max}), & d = 3. 
  \end{cases}
\end{displaymath}
Let $\MCs$ be the Cartesian grid over $\Omega^*$, and for simplicity,
we assume that $\MCs$ has the same grid spacing $s$ in all directions,
i.e. 
\begin{displaymath}
  \begin{aligned}
    s & = \frac{x_{\max} - x_{\min}}{n_x} = 
    \frac{y_{\max} - y_{\min}}{n_y}, && d = 2,  \\
    s &= \frac{x_{\max} - x_{\min}}{n_x}=
    \frac{y_{\max} - y_{\min}}{n_y} = \frac{z_{\max} - z_{\min}}{ n_z}
    , && d = 3,  \\
  \end{aligned}
\end{displaymath}
where $n_{x}(n_y, n_z)$ denote the numbers of elements in different
directions. 
We define $\MMs$ as the set of all nodes in $\MCs$.
For any $T \in \MCs$, we let $\mM_T := \{ \bvs \in \MMs: \ \bvs \in
\partial T \}$ consist of all vertices of $T$, and $T$ can be
described by two vertices $\bm{x}_{T, 1} = \{x_{T, 1}^j\}_{j = 1}^d,
\bm{x}_{T, 2} = \{ x_{T, 2}^j \}^d$ such that $\bm{x}_{T, 1} <
\bm{x}_{T, 2}$ with $T = \Pi_{j = 1}^d (x_{T, 1}^j, x_{T, 2}^j)$.
Throughout this paper, for two vectors $\bm{a}, \bm{b}$, the
inequality $\bm{a} <( \leq) \bm{b}$ is understood in a component-wise
manner.

Given any position $\bm{q} \in \Omega$, the index of the
host cell in $\MCs$ containing $\bm{q}$ can be rapidly determined
by its coordinates and the grid spacing $s$, which reads
\begin{equation}
  \begin{aligned}
    (\lfloor (q_x - x_{\min})/s \rfloor, \lfloor (q_y - y_{\min})/s
    \rfloor), \quad & \bm{q} = (q_x, q_y) \in \mb{R}^2, \\
    (\lfloor (q_x - x_{\min})/s \rfloor, \lfloor (q_y - y_{\min})/s
    \rfloor, \lfloor (q_z - z_{\min})/s \rfloor), \quad & \bm{q} =
    (q_x, q_y, q_z) \in \mb{R}^3,
  \end{aligned}
  \label{eq_Tindex}
\end{equation}
We note that only the Cartesian cells that have
intersection with $\Omega$ can be candidates of the host cell for
any $\bm{q} \in \Omega$.
These cells are called as active cells, and we denote by $\MCsc :=
\{T \in \MCs: \ |T \cap \Omega| > 0\}$ the set of all active cells. 

Finally, we define a distance function $d(\cdot, \cdot)$ such that
$d(\bm{w}_1, \bm{w}_2) = | \bm{w}_1 - \bm{w}_2|$ for any two points
$\bm{w}_1, \bm{w}_2 \in \mb{R}^d$. 
In addition, we let $d(\bm{w}, L)$ be the shortest distance
between $\bm{w}$ and $L$ for any point $\bm{w} \in \mb{R}^d$ and any
line $L \subset \mb{R}^d$.

\section{Particle \revise{Tracking}{Locating} based on Patch Searching} 
\label{sec_method}

\subsection{\revise{Tracking}{Locating} in two dimensions}
\label{subsec_track2d}
In this subsection, we present the algorithm in two dimensions.
Given any $\bm{p} \in \Omega$,
its host element $K \in \MTh$ is determined by two steps.
First we find out a node $\bv \in \MNh$ such that $\bm{p}$ is included
in the patch $\mD{\mT_{\bv}}$, and the second step is to seek $K$ in
$\mT_{\bv}$.

The node in the first step can be rapidly searched using the
background Cartesian grid. 
Here, the grid spacing $s$ is required to satisfy the
condition that $s \leq \frac{w^* \sin \alpha}{ \sqrt{2}(1 + \sin
\alpha)}$ with $w^* < w$,
where $\alpha$ comes from the minimum condition \eqref{eq_miniangle}.
The following lemma indicates every active cell can be associated with
a node $\bv$.

\begin{lemma}
  Under the condition $s \leq \frac{w^* \sin \alpha}{ \sqrt{2}(1 + \sin
  \alpha)}$, for any $T \in \MCsc$, there exists a node $\bv \in \MNh$
  such that $(T \cap \Omega) \subset \mD{\mT_{\bv}}$.
  \label{le_patch2d}
\end{lemma}
\begin{proof}
  According to whether the cell intersects an edge,  all active cells
  can be classified into two types. 
  We define 
  \begin{equation}
    \begin{aligned}
      \MCscI & := \{T \in \MCsc: \ \text{$\overline{T}$ does not
      intersect $\overline{e}$ for any edge $e \in \MEh$}\}, \\
      \MCscB & := \MCsc \backslash \MCscI. 
    \end{aligned}
    \label{eq_ccell}
  \end{equation}
  From the definition \eqref{eq_ccell}, any $T \in \MCscI$ will be
  entirely contained in an element $K \in \MTh$, see
  Fig.~\ref{fig_patch2d}. 
  Then, any vertex $\bv \in \mN_K$ can be associated to $T$.

  We further turn to the case that the cell $\overline{T}$ intersects 
  $\overline{e}$ at least for an edge 
  $e \in \MEh$.
  Let $\mN_{e} := \{\bv_0, \bv_1\}$, and 
  we let $\bm{q}$ be any point in the intersection between
  $\overline{T}$ and $\overline{e}$.
  If there exists $i$ such that $d(\bm{q}, \bv_i) \leq
  \frac{w^*}{1 + \sin \alpha}$, together with the diameter of $T$, we
  know that the distance between any vertex of $T$ and $\bv_i$ can be
  bounded by
  \begin{displaymath}
    d(\bvs, \bv_i) \leq d(\bm{q}, \bv_i) + \sqrt{2} s \leq
    \frac{w^*}{1 + \sin \alpha} + \frac{w^* \sin \alpha}{1 + \sin
    \alpha} = w^*, \quad \forall \bvs \in \mM_T,
  \end{displaymath}
  which implies $T \subset B(\bv_i, w^*)$ and thus $(T \cap \Omega)
  \subset (B(\bv_i, w^*) \cap \Omega) \subset \mD{\mT_{\bv_i}}$.
  It remains to prove for the case that
  \begin{equation}
    d(\bm{q}, \bv_i) > \frac{w^*}{1 + \sin \alpha}, \quad i = 0, 1.
    \label{eq_2d_1bvi}
  \end{equation}
  By Lemma \ref{le_app_disk}, we know that for $e \in \MEhi$,
  there holds $B(\bm{q}, \frac{w^* \sin
  \alpha}{1 + \sin \alpha}) \subset \mD{\mT_e}$.
  For any $\bvs \in \mM_T$, it is evident that $d(\bvs, \bm{q}) \leq
  \sqrt{2} s \leq  \frac{w^* \sin \alpha}{1 + \sin \alpha}$,
  which directly brings us that $T \subset B(\bm{q}, \frac{w^* \sin
  \alpha}{1 + \sin \alpha}) \subset \mD{\mT_e}$. 
  Clearly, $\mD{\mT_e} \subset \mD{\mT_{\bv_0}}$, and thus 
  $T$ can be associated with $\bv_0$.
  For $e \in \MEhb$, it is similar to find that $(B(\bm{q},
  \frac{w^* \sin \alpha}{1 + \sin \alpha}) \cap \Omega) \subset
  \mD{\mT_e}$. By Lemma \ref{le_app_disk}, $T$ can still be associated
  with the vertex $\bv_0$.
  This completes the proof.

  \begin{figure}[htp]
    \centering
    \begin{minipage}[t]{0.43\textwidth}
      \begin{center}
        \begin{tikzpicture}[scale=1.3]
          \footnotesize
          \draw[thick] (0, 0) -- (1, 0.3);
          \draw[thick] (0, 0) -- (0.4, 0.9);
          \draw[thick] (0, 0) -- (-0.6, 0.8);
          \draw[thick] (0.4, 0.9) -- (-0.6, 0.8) -- (-1, 0) -- (-0.65,
          -0.85) -- (0.65, -0.75) -- (1, 0.3) -- (0.4, 0.9);
          \draw[thick] (0, 0) -- (-1, 0);
          \draw[thick] (0, 0) -- (-0.65, -0.85);
          \draw[thick] (0, 0) -- (0.65, -0.75);
          \draw[thick, red] (-0.15, 0.5) rectangle (0.15, 0.8);
          \node[below right] at (0.07, 0.07) {$\bv$};
        \end{tikzpicture}
      \end{center}
    \end{minipage} \begin{minipage}[t]{0.43\textwidth}
      \begin{center}
        \begin{tikzpicture}[scale = 1.3]
          \footnotesize
          \draw[thick] (0, 0) -- (1.2, 0);
          \draw[thick] (0, 0) -- (0.8, -0.9);
          \draw[thick] (0, 0) -- (0.7, 0.9);
          \draw[thick] (0.8, -0.9) -- (1.2, 0) -- (0.7, 0.9);
          \draw[thick, red] (0.2, -0.15) rectangle (0.5, 0.15);
          \node[left] at (0, 0) {$\bv_0$};
          \node[right] at (1.2, 0) {$\bv_1$};
          \draw[dashed] (0.49, 0.63) arc [start angle = 60, end angle =
          -50, radius = 0.81];
          \node at (1.2, -0.5) {{\tiny $B(\bv_0, w^*)$}};
        \end{tikzpicture}
      \end{center}
    \end{minipage}

    \begin{minipage}[t]{0.43\textwidth}
      \begin{center}
        \begin{tikzpicture}[scale = 1.3]
          \footnotesize
          \draw[thick] (0, 0) -- (1.2, 0);
          \draw[thick] (0, 0) -- (0.8, -0.9);
          \draw[thick] (0, 0) -- (0.7, 0.9);
          \draw[thick] (0.8, -0.9) -- (1.2, 0) -- (0.7, 0.9);
          \draw[thick, red] (0.58, -0.15) rectangle (0.88, 0.15);
          \node[left] at (0, 0) {$\bv_0$};
          \node[right] at (1.2, 0) {$\bv_1$};
          \draw[dashed] (0.6, 0) circle [radius = 0.399];
          \node at (1.2, -0.56) {{\tiny$B(\bm{q}, \frac{w^* \sin
          \alpha}{1 + \sin \alpha})$}};
        \end{tikzpicture}
      \end{center}
    \end{minipage} \begin{minipage}[t]{0.43\textwidth}
      \begin{center}
        \begin{tikzpicture}[scale=1.3]
          \footnotesize
          \draw[thick, white] (0.5, 1.5) circle [radius=0.2];
          \draw[dashed] (-0.2, 0) rectangle (1.5, 1.2);
          \node at (1.25, 0.8) {$\Omega$};
          \draw[thick] (0, 0) -- (0.7, 0.9);
          \draw[thick] (1.2, 0) -- (0.7, 0.9);
          \node[left] at ($(0, -0.15) + (0.17, 0.35)$) {$\bv_0$};
          \node[right] at ($(1.2, -0.15) + (-0.1, 0.35)$) {$\bv_1$};
          \draw[thick, red] (0.58, -0.15) rectangle (0.88, 0.15);
          \draw[dashed] (0.6, 0) circle [radius = 0.399];
          \draw[thick, white, fill=white] (0, 0) -- (0.8, -0.9) -- (1.2, 0);
          \draw[thick] (0, 0) -- (1.2, 0);
          \node at (0.7, -0.22) {{\tiny$B(\bm{q}, \frac{w^* \sin
          \alpha}{1 + \sin \alpha})$}};
        \end{tikzpicture}
      \end{center}
    \end{minipage}
    \caption{The Cartesian cells in $\MCscI$ and $\MCscB$.}
    \label{fig_patch2d}
  \end{figure}
\end{proof}
Lemma \ref{le_patch2d} indicates that there exists a relation mapping
$\varphi: \MCsc \rightarrow \MNh$ such that $(T \cap \Omega) \subset
\mD{\mT_{\varphi(T)}}$ for any active $T \in \MCsc$.
The mapping $\varphi$ only depends on $\MTh$ and $\MCs$, which 
can be computed and stored before performing the
locating algorithm.
From $\varphi$, given any point $\bm{p} \in \Omega$, the associated
vertex $\bv$ can be fast localized using the Cartesian grid by $\bv
= \varphi(T)$ with $\bm{p} \in T$.

The next step is to seek the host element $K \in \mT_{\bv}$ for the
given $\bm{p} \in \mD{\mT_{\bv}}$, and we will show that 
this can be converted into a searching problem.
For $\bv \in \MNh$, we introduce a local coordinate system 
by letting $\bv$ be the origin, see Fig.~\ref{fig_2dlocalizepatch}.
Let $\mE_{\bv} = \{e_{\bv, 1}, e_{\bv, 2}, \ldots, e_{\bv, n}\}$ and
we let $\bm{v}_{\bv, i}$ be the unit vector along the edge $e_{\bv,
i}$ with the starting point $\bv$.
Let $\bm{\varepsilon} := (0, 1)^T$ be the unit vector on the $y$-axis,
and we let $\theta_{\bv, i}$ be the angle of rotating
$\bm{\varepsilon}$ to $\bm{v}_{\bv, i}$ in a clockwise direction. 
More details on computing such angles are given in Remark
\ref{re_angle}.
We further arrange the vectors $\{\bm{v}_{\bv, i} \}_{i=1}^n$ such
that $\{ \theta_{\bv, i} \}_{i=1}^n$ are sorted in ascending order. 
Because we have known that $\bm{p} \in \mD{\mT_{\bv}}$, seeking the
host element is equivalent to determining an index $i \in [1, n]$
such that $\bm{p}$ is included in the fan-shaped domain formed by
vectors $\bm{v}_{\bv, i}$ and $\bm{v}_{\bv, i+1}$, where $\bm{v}_{\bv,
n+1} := \bm{v}_{\bv, 1}$, see Fig.~\ref{fig_2dlocalizepatch}.
Let $\bm{v}$ be the vector connecting $\bv$ to $\bm{p}$, and we let
$\theta$ be the angle of clockwise rotating $\bm{\varepsilon}$ to the
direction of $\bm{v}$. 
Then, it suffices to search an index $i$ such that $\theta \in
(\theta_{\bv, i}, \theta_{\bv, i+1})$, where $\theta_{\bv, 0} = 0,
\theta_{\bv, n+1} := 2\pi$, which can be easily achieved by
the binary search algorithm.
From the minimum angle condition \eqref{eq_miniangle}, there holds $n
\leq \frac{2\pi}{\alpha}$, and the searching algorithm requires
$O(\log(\frac{2\pi}{\alpha}))$ comparisons. 
For any $\bv \in \MNh$, all angles $\{ \theta_{\bv, i} \}_{i = 1}^{n}$
can be prepared before performing the algorithm. 
For the node $\bv \in \MNhb$ lying on the boundary
$\partial \Omega$, there will be a fan-shaped
domain that is outside the domain $\Omega$. 
This domain can be simply marked with a flag $-1$ and the locating 
procedure is the same with interior nodes.  

\begin{figure}[htp]
  \centering
  \footnotesize
  \begin{minipage}[t]{0.43\textwidth}
    \begin{tikzpicture}[scale=1.35]
      \draw[thick, white, ->] (0, -1.3) -- (0, 1.3);
      \draw[thick, white] (-1.3, 0) -- (1.3, 0);
      \draw[thick] (0, 0) -- (1, 0.3);
      \draw[thick] (0, 0) -- (0.4, 0.9);
      \draw[thick] (0, 0) -- (-0.6, 0.8);
      \draw[thick] (0, 0) -- (-1, -0.15);
      \draw[thick] (0, 0) -- (-0.65, -0.85);
      \draw[thick] (0, 0) -- (0.65, -0.75);
      \draw[fill=red, red] (0.6, -0.3) circle [radius = 0.035];
      \draw[thick] (1, 0.3) -- (0.4, 0.9)-- (-0.6, 0.8)-- (-1,
      -0.15)-- (-0.65, -0.85)-- (0.65, -0.75) -- (1, 0.3);
      \node[below] at (0, -0.1) {$\bv$};
      \draw[black, fill=black] (0, 0) circle [radius=0.035];
      \node[right] at (0.4, 0.9) {$e_{\bv, 1}$};
      \node[right] at (1.0, 0.3) {$e_{\bv, 2}$};
      \node[right] at (0.65, -0.75) {$e_{\bv, 3}$};
      \node[left] at (-0.65, -0.85) {$e_{\bv, 4}$};
      \node[left] at (-1, -0.15) {$e_{\bv, 5}$};
      \node[left] at (-0.6, 0.8) {$e_{\bv, 6}$};
    \end{tikzpicture}
  \end{minipage}  \begin{minipage}[t]{0.43\textwidth}
    \begin{tikzpicture}[scale=1.35]
      \draw[thick, dashed, ->] (0, -1.3) -- (0, 1.3);
      \node[right] at (0, 1.2) {$\bm{\varepsilon}$};
      \draw[thick, white] (-1.3, 0) -- (1.3, 0);
      \draw[dashed, red] (0, 0) -- (0.6, -0.3);
      \draw[red] (0, 0.2) arc [start angle=90, end angle=-26,
      radius=0.2];
      \node[red] at (0.25, 0.25) {$\theta$};
      \draw[black] (0, 0.46) arc [start angle=90, end angle=18,
      radius=0.46];
      \node at (0.47, 0.47) {$\theta_{\bv, 2}$};
      \draw[black] (0, 0.75) arc [start angle=90, end angle=-172,
      radius=0.75];
      \node at (0.3, -0.85) {$\theta_{\bv, 5}$};
      \draw[thick, ->] (0, 0) -- (1, 0.3);
      \draw[thick, ->] (0, 0) -- (0.4, 0.9);
      \draw[thick, ->] (0, 0) -- (-0.6, 0.8);
      \draw[thick, ->] (0, 0) -- (-1, -0.15);
      \draw[thick, ->] (0, 0) -- (-0.65, -0.85);
      \draw[thick, ->] (0, 0) -- (0.65, -0.75);
      \draw[fill=red, red] (0.6, -0.3) circle [radius = 0.035];
      \draw[black, fill=black] (0, 0) circle [radius=0.035];
      \node[right] at (0.4, 0.9)    {$\bm{v}_{\bv, 1}$};
      \node[right] at (1.0, 0.3)    {$\bm{v}_{\bv, 2}$};
      \node[right] at (0.65, -0.75) {$\bm{v}_{\bv, 3}$};
      \node[left] at (-0.65, -0.85) {$\bm{v}_{\bv, 4}$};
      \node[left] at (-1, -0.15)    {$\bm{v}_{\bv, 5}$};
      \node[left] at (-0.6, 0.8)    {$\bm{v}_{\bv, 6}$};
    \end{tikzpicture}
  \end{minipage} \vspace{10pt}\begin{minipage}[t]{0.43\textwidth}
    \begin{tikzpicture}[scale=1.35]
      \draw[thick, white, ->] (0, -1.3) -- (0, 1.3);
      \draw[dashed] (-1.6, 0) rectangle (1.6, 1.8);
      \node at (0, -0.5) {flag: -1};
      \node[below left] at (1.2, 1.5) {$\Omega$};
      \node[below] at (0.6, 0) {$\partial \Omega$};
      \draw[thick] (-1, 0) -- (0, 0) -- (1, 0); 
      \draw[thick] (0, 0) -- (0.6, 0.8);
      \draw[thick] (0, 0) -- (-0.65, 0.82);
      \draw[thick] (-1, 0) -- (-0.65, 0.82) -- (0.6, 0.8) -- (1, 0);
      \node[below left] at (0, 0) {$\bv$};
      \draw[thick, fill=black] (0, 0) circle [radius=0.035];
      \node[right] at (0.6, 0.8)   {$e_{\bv, 1}$};
      \node[right] at (1, -0.15)   {$e_{\bv, 2}$};
      \node[left] at (-1, -0.15)   {$e_{\bv, 3}$};
      \node[left] at (-0.65, 0.82) {$e_{\bv, 4}$};
    \end{tikzpicture}
  \end{minipage} \begin{minipage}[t]{0.43\textwidth}
    \begin{tikzpicture}[scale=1.35]
      \draw[thick, white, ->] (0, -1.3) -- (0, 1.3);
      \draw[thick, dashed, ->] (0, -0.8) -- (0, 1.2);
      \draw[red] (0, 0.25) arc [start angle=90, end angle=23,
      radius=0.25];
      \node[red] at (0.30, 0.25) {$\theta$};
      \draw[] (0, 0.4) arc [start angle = 90, end angle = 55, radius
      = 0.4];
      \node[] at (0.22, 0.65) {$\theta_{\bv, 1}$};
      \draw[] (0, 0.75) arc [start angle = 90, end angle = -180, radius
      = 0.75];
      \node[] at (0.65, -0.7) {$\theta_{\bv, 3}$};
      \draw[red, dashed] (0, 0) -- (0.8, 0.3);
      \draw[thick, <->] (-1, 0) -- (0, 0) -- (1, 0);
      \draw[thick, fill=black] (0, 0) circle [radius=0.035];
      \draw[thick, ->] (0, 0) -- (0.6, 0.8);
      \draw[thick, ->] (0, 0) -- (-0.65, 0.82);
      \node[below left] at (0, 0) {$\bv$};
      \node at (0, -1) {flag: -1};
      \node[right] at (0.6, 0.8)   {$\bm{v}_{\bv, 1}$};
      \node[right] at (1, 0)       {$\bm{v}_{\bv, 2}$};
      \node[left] at (-1, 0)       {$\bm{v}_{\bv, 3}$};
      \node[left] at (-0.65, 0.82) {$\bm{v}_{\bv, 4}$};
      \draw[fill=red, red] (0.8, 0.3) circle [radius = 0.035];
    \end{tikzpicture}
  \end{minipage}
  \caption{localizing the point in the patch.}
  \label{fig_2dlocalizepatch}
\end{figure}
Given a point $\bm{p} \in \Omega$, our algorithm is summarized as
below:
\begin{equation}
  \bm{p} \xrightarrow{\step{1}} T \xrightarrow{\step{2}} \bv
  \xrightarrow{\step{3}} K.
  \label{eq_track2d}
\end{equation}
In \step{1} - \step{2}, the host cell $T$ can be easily found by
\eqref{eq_Tindex}, and then $\bv = \varphi(T)$.
In \step{3},
the host element $K$ is determined by computing the angle $\theta$
from  $\overrightarrow{\bv \bm{p}}$ and searching $\theta$ in
$\{\theta_{\bv, i} \}_{i =1}^n$.
The computational cost consists of three parts: locating on the
Cartesian grid, computing the pseudo angle from \eqref{eq_pangle} and
$O(\log(\frac{2\pi}{ \alpha}))$ comparisons.
Before performing the algorithm, our method needs to prepare
some data in the initializing stage, where 
the main step is to establish the relation mapping $\varphi$.
The computer implementation is detailed in next section.

\revise{}{
\begin{remark}
  In the local coordinate system with $\bv$ as the origin (see
  Fig.~\ref{fig_2dlocalizepatch}), computing the angles $\theta,
  \theta_{\bv, i} $ typically requires calling the inverse
  trigonometric
  function as $\arctan(\cdot)$ to vectors $\bm{v}$, $\bm{v}_{\bv,
  i}$.
  Then, seeking the host element converts into searching $\theta$ in
  $\{\theta_{\bv, i} \}$. 
  To have a better computational efficiency, we can use a simple
  function $\hat{\Theta}$, which is strictly positively related to the
  angle, to avoid precisely computing angles. 
  An example of such functions reads: for a given vector $\bm{w}$
  starting from the origin $(0, 0)^T$ to the point $(x, y)^T$, we
  define 
  \begin{equation}
    \hat{\Theta}(x, y) = \frac{\sign(x) x}{\sign(y)x + \sign(x)y} 
    - \sign(x)(\sign(y) + 1), \quad 
    \sign(z) = \begin{cases}
      1, & z \geq 0, \\
      -1, & z < 0, \\
    \end{cases}
    \label{eq_pangle}
  \end{equation}
  Using \eqref{eq_pangle}, we calculate $\hat{\theta}$ and
  $\{\hat{\theta}_{\bv, i}\}$ from vectors $\bm{v}$ and $\{\bm{v}_{\bv,
  i}\}$, respectively.
  It can be readily checked that localizing $\theta$ in
  $\{\theta_{\bv, i}\}$ is rigorously equivalent to localizing
  $\hat{\theta}$ in $\{\hat{\theta}_{\bv, i}\}$.
  Consequently, we can apply the function \eqref{eq_pangle} in the
  algorithm instead of precisely computing angles.
  From our tests, the function \eqref{eq_pangle} is numerically
  observed to be much faster than calling the function
  $\arctan(\cdot)$ for computing angles. 

  \label{re_angle}
\end{remark}
}

\subsection{\revise{Tracking}{Locating} in three dimensions}
\label{subsec_track3d}
In this section, we present the method in three dimensions, 
following a similar idea of two dimensions.
Given a point $\bm{p} \in \Omega$,  we first search a node $\bv \in
\MNh$ and the host element can be further found in the patch set
$\mT_{\bv}$.

The first step is also implemented with the background Cartesian
grid $\MCs$ with a specified spacing.
In three dimensions, the condition reads
\begin{equation}
  s \leq  \frac{2 w^* \sin \alpha \sin \frac{\alpha}{2}}{ \sqrt{3}( 1 +
  \sin \alpha)( 1 + \sin \frac{\alpha}{2})}, \quad w^* < \min(w,
  \frac{l_{\min}}{2}),
  \label{eq_s3d}
\end{equation}
where $l_{\min} := \min_{e \in \MEh} |e|$ denotes the length of the
shortest edge in the mesh $\MTh$.
Then, we demonstrate a similar result to Lemma \ref{le_patch2d}.

\begin{lemma}
  Under the condition \eqref{eq_s3d}, for any active $T \in \MCsc$,
  there exists a node $\bv \in \MNh$ such that $(T \cap \Omega)
  \subset \mD{\mT_{\bv}}$.
  \label{le_patch3d}
\end{lemma}
\begin{proof}
  For any $T \in \MCsc$, we let $B_T := B(\bom_{T}, r)$ be the
  circumscribed ball of $T$, where $\bom_{T}$ is the barycenter of $T$
  and $r = \frac{\sqrt{3} s}{2}$. 
  All active cells are classified into following types, which read
  \begin{equation}
    \begin{aligned}
      \MCscI &:= \{T \in \MCsc: \ \text{$\overline{T}$ does not
      intersect $\overline{f}$ for any
      face $f \in \MFh$} \}, \\
      \MCscF &:= \{T \in \MCsc \backslash \MCscI: \
      \text{$\overline{T}$ intersects at most one face $\overline{f}$
      for any element in $\MTh$}\}, \\
      \MCscE & := \{T \in \MCsc:  \ \text{$\overline{B}_T$ intersects
      $\overline{e}$ for an edge $e \in \MEh$} \}, \\
      \MCscB & := \MCsc \backslash (\MCscI \cup \MCscE \cup \MCscF).
    \end{aligned}
    \label{eq_cell3d}
  \end{equation}
  As in two dimensions, the proof for $T \in \MCscI$ is trivial.
  For any $T \in \MCscF$, we let $\overline{T}$ intersect
  $\overline{f}$ for a face $f \in \MFh$.
  For the case that $f \in \MFhi$ with $\mT_f = \{K_1, K_2\}$, 
  by the definition \eqref{eq_cell3d} the cell $T$ cannot intersect
  any other faces to both $K_1$ and $K_2$.
  This fact indicates that $T \subset (K_1 \cup K_2)$, and any vertex
  $\bv$ of $f$ can be picked up for $T$ with $T \subset
  \mD{\mT_{\bv}}$.
  For $f \in \MFhb$ with $\mT_f = \{K_1\}$, we know that $T$ cannot
  intersect any other faces of $K_1$, then $(T \cap \Omega) \subset
  K_1 \subset \mD{\mT_{\bv}}$ for any vertex $\bv$ of $f$.

  For $T \in \MCscE$, we let $\overline{B}_T$ intersect $\overline{e}$
  for an edge $e \in \MEh$ with $\mN_e = \{\bv_1, \bv_2\}$, 
  and we let $\bm{q}$ be any point in the intersection of
  $\overline{B}_T$ and $\overline{e}$.
  If there exists a vertex $\bv_i$ such that $d(\bm{q}, \bv_i) \leq
  \frac{w^*}{1 + \sin \alpha}$, we derive that 
  \begin{align*}
    d(\bzta, \bv_i) & \leq d(\bzta, \bm{q}) + d(\bm{q}, \bv_i) \leq 2r
    + \frac{w^* }{1 + \sin \alpha} \\
    & \leq \frac{2 w^* \sin \alpha \sin \frac{\alpha}{2} }{ (1 + \sin
    \alpha) (1 + \sin \frac{\alpha}{2})} +   \frac{w^*}{1 +
    \sin \alpha} \leq  \frac{w^* \sin \alpha}{1 + \sin \alpha}  +
    \frac{w^*}{1 + \sin \alpha} = w^*, \quad \forall \bzta \in T,  \\
  \end{align*}
  which indicates that $(T \cap \Omega) \subset \mD{\mT_{\bv_i}}$.
  If $d(\bm{q}, \bv_i) > \frac{w^* }{1 + \sin \alpha}$ for both
  vertices, we
  find that 
  \begin{equation}
    d(\bzta, \bm{q}) \leq 2r \leq  \frac{2 w^* \sin \alpha \sin
    \frac{\alpha}{2} }{ (1 + \sin \alpha) (1 + \sin \frac{\alpha}{2})}
    <  \frac{w^* \sin \alpha}{1 + \sin \alpha}, \quad \forall \bzta
    \in T.
    \label{eq_dzetaq}
  \end{equation}
  By Lemma \ref{le_app_disk}, there holds $(T \cap \Omega) \subset
  (B(\bm{q},  \frac{w^* \sin \alpha}{1 + \sin \alpha}) \cap \Omega)
  \subset \mD{\mT_{\bv_1}}$.

  For the last case that $T \in \MCscB$,
  from the definition \eqref{eq_cell3d} there exists an element $K$
  such that $\overline{T}$ intersects at least two faces $f_1$, $f_2$
  of $K$.
  Because $K$ is a tetrahedron, $f_1$ and $f_2$ will share a
  common edge $e \subset \partial K$.
  For both $i = 1, 2$, we know that the ball $B_T$ also intersects the
  face $f_i$.
  Let $\mO_i$ be the intersection between $\overline{B}_T$ and the
  plane along $f_i$, which is a disk on that plane with the center
  $\bc_i$.
  Since $T \not\in \MCscE$, $\overline{B}_T$ does not intersect any
  edge of $K$,
  which indicates that $\mO_i \subset f_i$ and $\bc_i \in f_i$.
  Let $\mN_K = \{ \bv_1, \bv_2, \bv_3, \bv_4 \}$ such that $\mN_{f_1}
  = \{ \bv_1, \bv_3, \bv_4 \}$, $\mN_{f_2} = \{\bv_1, \bv_2,
  \bv_3\}$, and $\mN_e = \{\bv_1, \bv_3\}$, see Fig.~\ref{fig_cell3d}.
  It can be observed that $\overrightarrow{\bom \bc_1} \perp f_1$ and
  $\overrightarrow{\bom \bc_2} \perp f_2$.
  Let $\bchi$ be the point on the line along $e$ such that
  $\overrightarrow{\bchi \bc_1} \perp \overrightarrow{\bv_1
  \bv_3}$.
  Then, $\bc_1, \bc_2, \bom, \bchi$ are
  all on the same plane, which is perpendicular to the vector
  $\overrightarrow{\bv_1\bv_2}$.
  We first consider the case that $\bchi$ is located in $e$, see the left
  figure in Fig.~\ref{fig_cell3d}.
  It is noted that the dihedral angle between $f_1$ and $f_2$ is the
  angle between $\overrightarrow{\bchi \bc_1}$ and
  $\overrightarrow{\bchi \bc_2}$, which is greater than $\alpha$ from
  \eqref{eq_miniangle}.
  Together with $d(\bom, \bc_i) \leq r$ for both $i = 1, 2$, we
  find that $d(\bom, \bchi) \leq \frac{r}{\sin \frac{\alpha}{2}}$.
  For any point $\bm{q} \in T$, there holds $d(\bm{q}, \bchi) \leq
  d(\bm{q}, \bom) + d(\bchi, \bom) \leq \frac{r}{\sin
  \frac{\alpha}{2}} + r$. 
  If $d(\bchi, \bv_i) \leq \frac{w^* }{1 + \sin \alpha}$ for $i =
  1$ or $i = 3$, the distance $d(\bq, \bv_i)$ can be estimated by
  \begin{align}
    d(\bm{q}, \bv_i) \leq d(\bm{q}, \bchi) + d(\bchi, \bv_i) & \leq
    \frac{r(1 + \sin \frac{\alpha}{2})}{ \sin \frac{\alpha}{2}} +
    \frac{w^*}{1 + \sin \alpha} \nonumber \\
    & \leq \frac{w^* \sin \alpha}{1 + \sin\alpha} + \frac{w^*}{1 +
    \sin \alpha} = w^*, \quad \forall \bq \in T, 
    \label{eq_dqbvi}
  \end{align}
  which immediately yields that $T \subset B(\bv_i, w^*)$ and $(T \cap
  \Omega) \subset \mD{\mT_{\bv_i}}$.
  If $d(\bchi, \bv_i) >
  \frac{w^*}{1 + \sin \alpha}$ for both $i = 1, 3$, by Lemma
  \ref{le_app_disk}, we have that $(B(\bchi, \frac{w^* \sin \alpha}{ 1
  + \sin \alpha}) \cap \Omega) \subset \mD{\mT_{\bv_i}}$. 
  For this case, the distance $d(\bq, \bv_i)$ can be estimated by
  \begin{equation}
    d(\bm{q}, \bchi) \leq \frac{r(1 + \sin \frac{\alpha}{2})}{\sin
    \frac{\alpha}{2}} \leq \frac{w^* \sin \alpha}{1 + \sin \alpha}, \quad
    \forall \bq \in T,
    \label{eq_dqchi}
  \end{equation}
  and thus $T \subset B(\bchi, \frac{w^* \sin \alpha}{ 1 + \sin
  \alpha})$.
  If $\bchi \not\in e$, see the right figure in Fig.~\ref{fig_cell3d},
  there exists $i \in [1, 2]$ such that
  $d(\bc_i, \bchi) \leq \frac{r}{\tan \frac{\alpha}{2}}$.
  Let $\bzta$ be the intersection point of the line along
  $\overrightarrow{\bc_i \bchi}$ and the line along
  $\overrightarrow{\bv_1 \bv_4}$.
  We have that $d(\bzta, \bchi) \leq d(\bc_i, \bchi)$ and 
  \begin{align}
    d(\bv_1, \bchi)  \leq  \frac{r \cot (2\alpha) }{\tan
    \frac{\alpha}{2}} &\leq \frac{ w^* \sin \alpha \sin \frac{\alpha}{2}
    \cos (2\alpha) }{ \tan \frac{\alpha}{2} \sin (2 \alpha) ( 1 + \sin
    \alpha) (1 + \sin \frac{\alpha}{2} )} \nonumber \\
    & \leq \frac{w^*}{1 + \sin \alpha} \frac{ \cos  \frac{\alpha}{2}
    \cos (2\alpha) }{ 2 \cos \alpha (1 + \sin  \frac{\alpha}{2})  } <
    \frac{w^*}{1 + \sin \alpha}.
    \label{eq_dv1chi}
  \end{align}
  The last inequality follows from the estimate $\cos(2 \beta) < 2
  \cos \beta$ for $\forall \beta \in [0, \frac{\pi}{3}]$ and $\alpha
  \leq  \frac{\pi}{3}$.
  For any $\bm{q} \in T$, the distance $d(\bm{q}, \bv_1)$ can be
  bounded as \eqref{eq_dqbvi}, which implies $(T \cap \Omega)
  \subset \mD{\mT_{\bv_1}}$.
  This completes the proof.

  \begin{figure}[htp]
    \centering
    \begin{minipage}[t]{0.46\textwidth}
      \begin{center}
        \begin{tikzpicture}[scale=1.8]
          \footnotesize
          \coordinate (A) at (0, 0);
          \coordinate (B) at (1.35, -0.8);
          \coordinate (C) at (2, 0);
          \coordinate (D) at (0.8, 1.16);
          \coordinate (C0) at (0.55, 0.2);
          \coordinate (C1) at (0.9, -0.25);
          \coordinate (O) at (0.9, 0.05);
          \coordinate (W) at (0.55, -0.326);
          \draw[thick] (A) -- (B) -- (C) -- (D) -- (A);
          \draw[thick] (B) -- (D);
          \draw[dashed] (A) -- (C);
          \draw[thick, fill=black] (C0) circle [radius=0.015];
          \draw[dashed, rotate=65] (C0) ellipse (2.5mm and 1.3mm);
          \draw[thick, fill=black] (C1) circle [radius=0.015];
          \draw[thick, fill=black] (O) circle [radius=0.015];
          \draw[thick, fill=black] (W) circle [radius=0.015];
          \draw[dashed] (C1) ellipse (2mm and 1mm);
          \draw[dashed] (C0) -- (O) -- (C1);
          \draw[dashed] (C0) -- (W) -- (C1);
          \draw[dashed] (O) -- (W);
          \node[left] at (A)  {$\bv_1$};
          \node[right] at (C) {$\bv_2$}; 
          \node[below] at (B) {$\bv_3$}; 
          \node[above] at (D) {$\bv_4$}; 
          \node[above left] at (C0) {$\bc_1$};
          \node[right] at ($(C1) - (0, 0.02)$) {$\bc_2$};
          \node[right] at (O) {$\bom$};
          \node[below] at (W) {$\bchi$};
          \node at (1.5, -0.3) {$f_2$};
          \node at (0.75, 0.8) {$f_1$};
        \end{tikzpicture}
      \end{center}
    \end{minipage}    \begin{minipage}[t]{0.46\textwidth}
      \begin{center}
        \begin{tikzpicture}[scale=1.8]
          \footnotesize
          \coordinate (A) at (0, 0);
          \coordinate (B) at (1.35, -0.4);
          \coordinate (C) at (0.35, 0.9);
          \coordinate (D) at (-0.8, 1.46);
          \coordinate (C0) at (-0.15, 0.5);
          \coordinate (C1) at (0.25, 0.39);
          \coordinate (O) at (0.9, 0.05);
          \coordinate (W) at (-0.1, 0.02);
          \coordinate (G) at (-0.13, 0.24);
          \draw[thick] (A) -- (B) -- (C) -- (D) -- (A);
          \draw[thick] (B) -- (D);
          \draw[dashed] (A) -- (C);
          \draw[dashed] (-0.35, 0.1) -- (1.7, -0.5);
          \node[left] at ($(A)+ (0.3, -0.13)$)  {$\bv_1$};
          \node[right] at (C) {$\bv_2$}; 
          \node[below] at (B) {$\bv_3$}; 
          \node[above] at (D) {$\bv_4$}; 
          \draw[thick, fill=black] (C0) circle [radius=0.015];
          \draw[thick, fill=black] (C1) circle [radius=0.015];
          \draw[thick, fill=black] (W) circle [radius=0.015];
          \draw[thick, fill=black] (G) circle [radius=0.015];
          \node[above right] at (C0) {$\bc_1$};
          \node[left] at (G) {$\bzta$};
          \node[below right] at (C1) {$\bc_2$};
          \node[below left] at (W) {$\bchi$};
          \draw[dashed] (C0) -- (W) -- (C1);
        \end{tikzpicture}
      \end{center}
    \end{minipage}
    \caption{The tetrahedron $K$ and the points $\bc_1,
    \bc_2, \bchi, \bom$.}
    \label{fig_cell3d}
  \end{figure}
\end{proof}
From Lemma \ref{le_patch3d}, in three dimensions there still exists a
relation mapping $\varphi: \MCsc \rightarrow \MNh$, where $(T \cap
\Omega) \in \mD{\mT_{\varphi(T)}}$ for any $T \in \MCsc$.
Hence, for any position $\bm{p} \in \Omega$, the node $\bv$ with
$\bm{p} \in \mD{\mT_{\bv}}$ can be rapidly determined by the Cartesian
grid $\MCs$.
In the implementation, the mapping $\varphi$ can be prepared in the
initialization stage.

Now, the second step is to find the host element in
$\mT_{\bv}$ for the given $\bm{p} \in \mD{\mT_{\bv}}$. 
This step is different from the two-dimensional case because 
elements in $\mT_{\bv}$ cannot be related to a series of angles in
three dimensions.
Thus, the procedure of localizing the corresponding angle to find the
host element does not work for elements in $\mT_{\bv}$.
But we will show that this procedure can be applied to elements in the
patch $\mT_{e}$ for any edge $e \in \MEh$.
We notice that in three dimensions, the host cell $T$ of the given
point $\bm{p}$ may not be included in any patch $\mD{\mT_e}$, see
Remark \ref{re_movingstep}. 
We further propose a moving step to seek an edge $e$ such that the
host element of $\bm{p}$ lies in $\mT_e$.

\revise{
and we define $\MCMEhc :=
\bigcup_{e \in \MEh} \MCec$.
}{Let us give some notation. 
For any edge $e \in \MEh$, we let $\MCec := \{T \in \MCsc: \ (T \cap
\Omega) \subset \mD{\mT_e} \}$ be the set of all active Cartesian
cells located in the patch $\mD{\mT_{e}}$.
This definition indicates that any cell $T \in \MCec$ can be
associated with $e$ in the sense that $(T \cap \Omega) \in
\mD{\mT_e}$.
Let $\MCMEhc := \bigcup_{e \in \MEh} \MCec$, 
and thus, any cell in $\MCMEhc$ can at least be associated 
with an edge. 
For any node $\bv \in \MNh$, we define $\MCBvc := \{T \in \MCsc: \
|\overline{T} \cap \partial B(\bv, w^*)| > 0 \}$ as the set formed by
all cells intersecting with the sphere $\partial B(\bv, w^*)$. 
For the given point $\bm{p}$ with the associated node $\bv$, i.e.
$\bm{p} \in \mD{\mT_{\bv}}$, we 
will construct a new point $\wt{\bm{p}}$ on $\partial
B(\bv, w^*)$. 
Clearly, the host cell $\wt{T}$ of $\wt{\bm{p}}$ belongs to $\MCBvc$.
In the following lemma, we show that any cell $\wh{T} \in \MCBvc$ can
be associated with an edge, i.e. $\wh{T} \in \MCMEhc$.
Consequently, the host cell of $\wt{p}$ can be associated with an edge.
}

\begin{lemma}
  Under the condition \eqref{eq_s3d}, there holds
  $\mC_{B_{\bv}}^{\circ} \subset \mC_{\MEh}^{\circ}$ for $\forall \bv
  \in \MNh$.
  \label{le_celledge3d}
\end{lemma}
\begin{proof}
  As the proof of Lemma \ref{le_patch3d}, the set $\MCsc$ is still
  decomposed into several categories as \eqref{eq_cell3d}.
  Clearly, there holds $T \in \mC_{\MEh}^{\circ}$ if $T \in \MCscI$ or
  $T \in \MCscF$ for any $T \in \mC_{B_{\bv}}^{\circ}$.

  For the case that $T \in \MCscE$, there exist a node $\bv_1$ and an
  edge $e$ such that $\overline{T}$ intersects both the sphere
  $\partial B(\bv_1, w^*)$ and $\overline{e}$.
  Let $\bzta$ be any point in the intersection between $\overline{T}$
  and $\overline{e}$. 
  If $\bv_1 \not\in \mN_e$, we know that $d(\bzta, \bv) > w^*$ for
  both $\bv \in \mN_e$. 
  By \eqref{eq_dzetaq}, we obtain that $(T \cap \Omega) \subset
  (B(\bm{q}, \frac{w^* \sin \alpha}{1 + \sin \alpha}) \cap \Omega)
  \subset \mD{\mT_e}$.
  If $\bv_1 \in \mN_e$, we let $\mN_e = \{ \bv_1, \bv_2\}$, 
  and we let 
  $\bm{q} \in \overline{T}$ such that $d(\bm{q}, \bv_1) = w^*$.
  The distance $d(\bzta, \bv_1)$ can be estimated as below,
  \begin{align*}
    d(\bzta, \bv_1) \geq d(\bq, \bv_1) - d(\bzta, \bq) > w^* -
    \frac{2 w^* \sin \alpha \sin \frac{\alpha}{2}}{ ( 1 + \sin
    \alpha)( 1 + \sin \frac{\alpha}{2})} > w^* - \frac{w^* \sin \alpha
    }{ ( 1 + \sin \alpha)} = \frac{w^*}{1 + \sin \alpha}. 
  \end{align*}
  From the triangle inequality, 
  there holds $d(\bq, \bv_2) \geq d(\bv_1, \bv_2) - d(\bq,
  \bv_1) \geq w^*$, 
  and it is similar to obtain that
  \begin{displaymath}
    d(\bzta, \bv_2) \geq d(\bq, \bv_2) - d(\bzta, \bq) >
    \frac{w^*}{1 + \sin \alpha}.
  \end{displaymath}
  Again by \eqref{eq_dzetaq}, we conclude that
  $(T \cap \Omega) \subset \mD{\mT_e}$.

  We last consider the case $T \in \MCscB$. 
  As the proof of Lemma \ref{le_patch3d}, we also 
  let $\overline{T}$ intersects two faces $f_1$ and $f_2$, which share
  a common edge $e$, 
  and we let $\bv_1, \bv_2, \bv_3, \bv_4, \bc_1, \bc_2, \bchi, \bom$
  be the same as in Fig.~\ref{fig_cell3d}.
  If $\bchi \not\in e$, by \eqref{eq_dv1chi} and \eqref{eq_dqchi} there holds $d(\bq,
  \bv_1) < w^*$ for $\forall \bq \in T$, which gives that 
  $T \subset B(\bq, w^*)$. 
  It is noticeable that the condition $w^* < \frac{l_{\min}}{2}$
  indicates that $B(\bv^+, w^*) \cap B(\bv^-, w^*) = \varnothing$ for
  $\forall \bv^+, \bv^- \in \MNh$. 
  Since $T \subset B(\bv_1, w^*)$, we conclude that $\overline{T} \cap
  \partial B(\bv^+, w^*) = \varnothing$ for $\forall \bv^+ \in \MNh$, 
  which violates the definition of $\MCBvc$. 
  It suffices to consider that $\bchi \in e$. 
  If $T \in \mC_{B_{\bv_j}}^{\circ}$ for $j = 2$ or $j = 4$,
  we let $\bq \in \overline{T}$ such that $d(\bq, \bv_j) = w^*$.
  From the triangle inequality and \eqref{eq_dqchi}, we find that
  $d(\bq, \bv_i) \geq d(\bv_i, \bv_j) - d(\bq, \bv_j)
  > w^*$ for both $i = 1, 3$, then 
  there holds
  \begin{align*}
    d(\bv_i, \bchi) \geq d(\bq, \bv_i) - (d(\bchi, \bom) + d(\bom,
    \bq)) > w^* - \frac{w^* \sin \alpha}{1 + \sin \alpha} =
    \frac{w^*}{1 + \sin \alpha}.
  \end{align*}
  By \eqref{eq_dqchi}, we obtain that $T \subset B(\bchi,
  \frac{w^* \sin \alpha}{1 + \sin \alpha})$ and $(T \cap \Omega)
  \subset \mD{\mT_e}$.
  If $T \in \mC_{B_{\bv_i}}^{\circ}$ for $i = 1$ or $i = 3$, without
  loss of generality, we assume $i = 1$.
  We derive that
  \begin{align*}
    d(\bchi, \bv_1) & \geq d(\bm{q}, \bv_1) - (d(\bm{q}, \bom) +
    d(\bom, \bchi)) \geq w^* - \frac{w^* \sin \alpha}{1 + \sin \alpha}
    = \frac{w^*}{1 + \sin \alpha}.
  \end{align*}
  By the triangle inequality, we have that $d(\bm{q}, \bv_3) \geq
  d(\bv_1, \bv_3) - d(\bm{q}, \bv_1) \geq \frac{l_{\min}}{2} >
  w^*$. 
  It is similar to deduce that $d(\bchi, \bv_3) \geq
  \frac{w^*}{1 + \sin \alpha}$. 
  Again by the estimate \eqref{eq_dqchi}, there holds $(T \cap
  \Omega) \subset \mD{\mT_e}$, which completes the proof.
\end{proof}

\revise{}{
Lemma \ref{le_celledge3d} directly indicates 
$\MCNhsc \subset \MCMEhc$, where
$\MCNhsc := \bigcup_{\bv \in \mN_h} \MCBvc$,
and also implies that there exists a relation mapping $\phi: \MCNhsc
\rightarrow \MEh$ such that  $(T \cap \Omega) \subset
\mD{\mT_{\phi(T)}}$ for any $T \in \MCNhsc$.

In the computer implementation given in Section \ref{sec_implement}, 
we actually construct the mapping $\phi$ on a larger set $\MCNhssc$ in
the sense that $\MCNhsc \subset \MCNhssc \subset \MCMEhc$.
Generally speaking, this construction will slightly accelerate the
locating algorithm, and more importantly, both the set $\MCNhssc$ and
the mapping $\phi$ can be obtained in the process of constructing the
mapping $\varphi$, which is easier than precisely identifying all
cells of $\MCNhsc$ and the associated edges. 
The mapping $\phi$ can be extended to $\MCsc$ by formally setting
$\phi(T) = -1$ for any $T \in \MCsc \backslash \MCNhssc$.

}

Now, let us present the moving step. 
Given $\bm{p} \in \mD{\mT_{\bv}}$ with its host cell $T
\in \MCsc$, 
if $\phi(T) = -1$, i.e. $T \not\in \MCNhssc$, 
we define a new point $\wt{\bm{p}}$ by
\begin{equation}
  \wt{\bm{p}} = \bv + \frac{w^*}{d(\bm{p}, \bv)}(\bm{p} - \bv).
  \label{eq_newp}
\end{equation}
A direct calculation is that $d(\wt{\bm{p}}, \bv) = {w^*} < w$,
which brings that $\bm{p}$ and $\wt{\bm{p}}$ have the same host
element in $\mT_{\bv}$.
From $\wt{\bm{p}}$, we also find out its host Cartesian cell $\wt{T}
\in \MCsc$ and set $\wt{\bv} = \varphi(\wt{T})$.
The definition of $\MCNhsc$ brings that $\wt{T} \in \MCNhssc$.
If $\phi(T) \neq -1$, i.e. $T \in \MCNhssc$,  we can directly let
$\wt{T} = T$, $\wt{\bv} = \bv$.
Consequently, we arrive at
$\wt{T} \in \MCNhssc$, and by the mapping $\phi$, we know that $\wt{e}
= \phi(\wt{T})$ satisfying $\wt{\bm{p}} \in \mD{\mT_{\wt{e}}}$.

\revise{}{
\begin{remark}
  For the given point $\bm{p}$, the main purpose of the moving step
  \eqref{eq_newp} is to 
  construct a new point $\wt{\bm{p}}$, which shares the same host
  element but is far away from all vertices in the mesh.
  For such a new point, its host Cartesian cell will be contained in a
  patch near an edge.
  If $\bm{p}$ is very close to a vertex $\bvs$, 
  its host cell $T$ will contain $\bvs$ in its interior. 
  In this case, $T$ crosses all the edges having $\bvs$ as an
  endpoint, and is not contained in any patch $\mD{\mT_e}$.
  Here we briefly show that the new point $\wt{\bm{p}}$ is distant
  from all vertices in the mesh. 
  Let $K$ be the host element for $\bm{p}$ and $\wt{\bm{p}}$. 
  Then, $\bv = \varphi(T)$ is a vertex of $K$, and we let $\bv_1, \bv_2, \bv_3$
  be the other vertices of $K$.
  Since $d(\wt{\bm{p}}, \bv) = w^*$, from the triangle
  inequality, we find that
  \begin{displaymath}
    d(\wt{\bm{p}}, \bv_j) \geq d(\bv, \bv_j) - d(\wt{\bm{p}}, \bv)
    \geq l_{\min} - w^* > w^*, \quad j = 1,2,3.
  \end{displaymath}
  For any $\wh{\bv} \in \MNh$ that is not the vertex of $K$,
  there holds $d(\wt{\bm{p}}, \wh{\bv}) \geq w > w^*$.
  We thus conclude that $\wt{\bm{p}}$ is far away from all vertices.
  Then, its host cell $\wt{T}$ can be contained in
  a patch $\mD{\mT_e}$.
  \label{re_movingstep}
\end{remark}
}

In the final step, the task turns into determining the host element in $\mT_{e}$
for a given $\bm{q} \in \mD{\mT_{e}}$, which can be
implemented by locating angles as the second step in two dimensions.
For the edge $e \in \MEh$ with vertices $\bv_1, \bv_2$, 
we fix $\bv_e = \bv_1$ and
let $P_e$ be the plane that is orthogonal to the vector
$\overrightarrow{\bv_{1} \bv_{2}}$ and passes through the vertex
$\bv_e$, see Fig.~\ref{fig_MCse}.
\revise{}{
In $P_e$, each element in $\mT_e$ will correspond to a sector.}
For any $K \in \mT_{e}$, we let $e_{K, 1}, e_{K, 2}, e_{K, 3}$ be
the three edges of $K$ sharing the common vertex $\bv_e$, and let
$e_{K, 1} = e$.
We define the set $\mE_{\bv_e} := \bigcup_{K \in \mT_{e}} \bigcup_{j =
2, 3} e_{K, j}$, which is further rewritten as $\mE_{\bv_e} = \{e_{1},
e_{2}, e_3, \ldots, e_m\}$, see Fig.~\ref{fig_MCse}.
Let $\bm{v}_{e, j}$ be the unit vector along the edge $e_j$ with the
starting point $\bv_e$ and let $\bm{w}_{e, j}$ be the projection of
$\bm{v}_{e, j}$ onto the plane $P_e$, and every $\bm{w}_{e, j}$ is
further scaled to be a unit vector still with the starting point
$\bv_e$.
\revise{}{
In $P_e$, we establish a two-dimensional local coordinate
system with $\bv_e$ as the origin. 
Let $\bm{\varepsilon} = (0,1)^T$ be the unit vector, 
and let $\theta_{e, j}$ be the angle of clockwise rotating
$\bm{\varepsilon}$ to the vector $\bm{w}_{e, j}$. 
Next, we arrange the edges $\{e_j \}_{j = 1}^m$ and the
corresponding vectors $\{ \bm{w}_{e, j} \}_{j = 1}^m$ such that
$\{\theta_{e, j}\}_{j = 1}^m$ are placed in ascending order.
By this rearrangement, it is noted that the sector shaped by vectors
$\bm{w}_{j}$ and $\bm{w}_{j+1}$ in the plane $P_e$ corresponds to 
a unique element in $\mT_e$, which has two edges with respect to
$\bm{w}_j$ and $\bm{w}_{j+1}$, see Fig.~\ref{fig_MCse}.
Let $\bm{v} := \overrightarrow{\bv_e \bm{q}}$ be the vector connecting
$\bv_e$ to $\bm{q}$, and we compute the projection of $\bm{v}$
onto the plane $P_e$ as $\bm{w}$. 
The task of locating $\bm{q}$ in $\mD{\mT_e}$ now becomes 
locating the vector $\bm{w}$ in the plane $P_e$, which is the same as
two dimensions.
Let $\theta$ be the angle of clockwise rotating $\bm{\varepsilon}$ to
the direction of $\bm{w}$. 
It remains to seek $\theta$ in $\{\theta_{e, j} \}_{j=1}^m$, and this
can be readily achieved by the standard searching algorithm. 
For every edge $e$, the plane $P_e$ and the angles $\{\theta_{e, j}
\}_{j=1}^m$ and the relations between sectors and elements can be
prepared in the initializing stage. 
}

\begin{figure}[htp]
  \centering
  \begin{minipage}[t]{0.4\textwidth}
    \footnotesize
    \begin{tikzpicture}[scale=2.2]
      \coordinate (v2) at (0, 0);
      \coordinate (v1) at (1.5, 2);
      \coordinate (A) at (-0.1, 1.35);
      \coordinate (B) at (0.6, 0.6);
      \coordinate (C) at (1.5, 0.9);
      \coordinate (D) at (0.7, 1.3);
      \coordinate (E) at (1.2, 0.6);
      \draw[dashed] (v2) -- (v1);
      \draw[thick] (v2) -- (A) -- (v1);
      \draw[thick] (v2) -- (B) -- (v1);
      \draw[dashed] (v2) -- (C);
      \draw[thick] (C) -- (v1);
      \draw[thick] (v2) -- (E) -- (v1);
      \draw[dashed] (v2) -- (D) -- (v1);
      \draw[thick] (A) -- (B) -- (E) -- (C);
      \draw[dashed] (A) --(D) -- (C);
      \node[right] at (v1) {$\bv_1(\bv_e)$};
      \node[left] at (v2) {$\bv_2$};
      \node at (0.65, 1) {$e$};
      \node[left] at (A) {$A_1$};
      \node[below right] at ($(B) - (0.1, 0)$) {$A_2$};
      \node[right] at (C) {$A_4$};
      \node[right] at ($(E) - (0, 0.05)$) {$A_3$};
      \node[above] at (D) {$A_5$};
      \node[below] at ($(A) + (0.65, 0.55)$) {$e_1(\bm{v}_{e, 1})$};
      \node[below] at ($(B) + (0.39, 0.30)$) {$e_2(\bm{v}_{e, 2})$};
      \node[below] at ($(C) + (0.06, 0.26)$) {$e_3(\bm{v}_{e, 3})$};
      \node[below] at ($(C) + (0.26, 0.65)$) {$e_4(\bm{v}_{e, 4})$};
      \node[below] at ($(D) + (0.25, 0.42)$) {$e_5(\bm{v}_{e, 5})$};
      \draw[fill=black] (v1) circle [radius = 0.025];
      \draw[fill=black] (v2) circle [radius = 0.025];
      \draw[fill=black] (A) circle [radius = 0.025];
      \draw[fill=black] (B) circle [radius = 0.025];
      \draw[fill=black] (C) circle [radius = 0.025];
      \draw[fill=black] (D) circle [radius = 0.025];
      \draw[fill=black] (E) circle [radius = 0.025];
    \end{tikzpicture}
  \end{minipage} \begin{minipage}[t]{0.4\textwidth}
    \footnotesize
    \begin{tikzpicture}[scale=2.2]
      \coordinate (v1) at (1.5, 2);
      \coordinate (A) at (-0.1, 1.35);
      \coordinate (B) at (0.6, 0.6);
      \coordinate (C) at (1.5, 0.9);
      \coordinate (D) at (0.7, 1.3);
      \coordinate (E) at (1.2, 0.6);
      \coordinate (p0) at (0.3, 1.8);
      \coordinate (p1) at (2, 1.2);
      \coordinate (p2) at (2.6, 2.1);
      \coordinate (p3) at (0.9, 2.6);
      \coordinate (a0) at (1.1, 2.05);
      \coordinate (a1) at (1.35, 2.3);
      \coordinate (a2) at (1.9, 2.);
      \coordinate (a3) at (1.6, 1.6);
      \coordinate (a4) at (1.26, 1.7);
      \draw[red, dashed, thick] (1.62, 2.1) to [out = -30, in = 20]
      (1.57, 1.8);
      \draw[dashed, ->] (v1) -- (A); 
      \draw[dashed, ->] (v1) -- (B); 
      \draw[dashed, ->] (v1) -- (C); 
      \draw[dashed, ->] (v1) -- (0.5, 1.1); 
      \draw[dashed, ->] (v1) -- (E); 
      \draw[thick] (p0) -- (p1) -- (p2) -- (p3) -- (p0);
      \node at (2.39, 1.5) {$P_e$};
      \draw[thick, ->] (v1) -- (a0);
      \draw[thick, ->] (v1) -- (a1);
      \draw[thick, ->] (v1) -- (a2);
      \draw[thick, ->] (v1) -- (a3);
      \draw[thick, ->] (v1) -- (a4);
      \node[below] at ($(A) + (0.55, 0.20)$)  {$\bm{v}_{e, 1}$};
      \node[below] at ($(B) + (0.36, 0.37)$)  {$\bm{v}_{e, 2}$};
      \node[below] at ($(C) + (-0.08, 0.03)$) {$\bm{v}_{e, 3}$};
      \node[below] at ($(C) + (0.13, 0.42)$)  {$\bm{v}_{e, 4}$};
      \node[below] at ($(D) + (0.28, 0.15)$)  {$\bm{v}_{e, 5}$};
      \node[left] at (a0)                            {$\bm{w}_{e, 1}$};
      \node[above] at ($(a1) - (0, 0.035)$)          {$\bm{w}_{e, 5}$};
      \node[right] at (a2)                           {$\bm{w}_{e, 4}$};
      \node[below right] at ($(a3) + (-0.05, 0.03)$) {$\bm{w}_{e, 3}$};
      \node[below] at ($(a4) + (-0.07, 0.03)$)       {$\bm{w}_{e, 2}$};
      \node at ($(v1) + (-0.13, 0.08)$) {$\bv_e$};
      \draw[black, fill=black] (v1) circle [radius=0.02];
      \draw[dashed, thick, ->] (v1) -- (1.7, 2.3);
      \node at (1.78, 1.82) {\tred{$\theta_{e, 3}$}};
      \node at (1.75, 2.2) {$\bm{\varepsilon}$};
    \end{tikzpicture}
  \end{minipage}
  \caption{The set $\mT_{e}$ and the plane $P_e$.}
  \label{fig_MCse}
\end{figure}

For a given point $\bm{p} \in \Omega$, finding its host element in
three dimensions is summarized as: 
\begin{equation}
  \bm{p} \xrightarrow{\step{1}} T \xrightarrow{\step{2}} \bv
  \xrightarrow{\step{3}}
  \wt{\bm{p}} \xrightarrow{\step{4}} \wt{T} \xrightarrow{\step{5}} e
  \xrightarrow{\step{6}} \revise{}{P_e \xrightarrow{\step{7}} K}.
  \label{eq_track3d}
\end{equation}
In steps \step{1} - \step{3}, we seek the host cell $T \in \MCsc$ and
let $\bv = \varphi(T)$ by the Cartesian grid. 
If $\phi(T) = -1$, we construct $\wt{\bm{p}}$ as \eqref{eq_newp}, and
find $\wt{T}$ as its host cell and let $e = \phi(\wt{T})$ for steps
\step{4} and \step{5}.
The last two steps \step{6} and \step{7} are to determine the host
element in $\mT_e$ for $\wt{\bm{p}}$.
The computational cost mainly consists of several parts: 
localizing on the Cartesian grid twice, 
constructing a new point as \eqref{eq_newp}, 
computing the pseudo angle by \eqref{eq_pangle} and $
O(\log(\frac{2\pi}{\alpha}))$ comparisons. 
In the preparation stage, we are required to establish two relation
mappings $\varphi$, $\phi$ in three dimensions.
The details of the computer implementation for mappings are presented
in next section.

\section{Computer Implementation}
\label{sec_implement}
In this section, we present details on the computer
implementation to the proposed \revise{tracking}{locating} method, 
particularly for the preparation stage. 

We start from two dimensions.
As stated in Subsection \ref{subsec_track2d}, 
the main step of the initialization step is to establish the relation
mapping $\varphi: \MCsc \rightarrow \MNh$ such that $T \in
\mD{\mT_{\varphi(T)}}$ for any $T \in \MCsc$. The construction to
$\varphi$ follows the idea
in the proof to Lemma \ref{le_patch2d}.
By the definition \eqref{eq_ccell}, for any $T \in \MCsc$,
$\varphi(T)$ can be obtained by
checking whether $\overline{T}$ is cut by an edge. 
Here, we construct $\varphi$ by finding out all Cartesian
cells that intersect with $\overline{e}$ for each $e \in \MEh$.
For any $e \in \MEh$ with $\mN_e = \{ \bv_1, \bv_2 \}$, let $T_1,
T_2 \in \MCsc$ be the host cells for $\bv_1, \bv_2$ indexed by $(i_1,
j_1), (i_2, j_2)$, respectively. 
Let 
\begin{equation}
  (i_{e, 1}, j_{e, 1}) = (\min(i_1, i_2), \min(j_1, j_2)), \quad
  (i_{e, 2}, j_{e, 2}) = (\max(i_1, i_2), \max(j_1, j_2)), 
  \label{eq_boxdomain}
\end{equation}
and we construct a bound box $\mB_e$ containing all Cartesian cells 
indexed by $(i_l, j_l)$ for $\forall (i_l, j_l) \in [i_{e, 1}, i_{e,
2}] \times [j_{e, 1}, j_{e, 2}]$, see left figure in
Fig.~\ref{fig_box}.
It is noted that all Cartesian cells that are cut by the edge
$\overline{e}$ are
also included in $\mB_e$. 
For every cell $T \in \mB_e$, $T$ can be described by two vertices
$\bm{x}_{T, 1}$, $\bm{x}_{T, 2}$ with $\bm{x}_{T, 1} < \bm{x}_{T, 2}$.
Then, $\overline{T}$ and $\overline{e}$ intersecting is equivalent to
the inequality $\bm{x}_{T, 1} \leq \bv_1 + t (\bv_2 - \bv_1) \leq
\bm{x}_{T, 2}$ admits a solution $t \in [0, 1]$, which can be easily
solved. 
For $T \in \mB_e$, if such $t$ exists, we let $\bm{q} := \bv_1 + t
(\bv_2 - \bv_1)$ be a point in the intersection of $\overline{T}$ and
$\overline{e}$, and we know that $T \in \MCscB$.
From the proof of Lemma \ref{le_patch2d}, the vertex of $e$ that is
closest to $\bm{q}$ can be the associated node $\varphi(T)$ for $T$. 
By this procedure, all cells that are cut by $e$ have been associated
with corresponding nodes.
Then, the relation mapping $\varphi$ on $\MCscB$ can be constructed in
a piecewise manner. 

We next turn to the set $\MCscI$,  where any cell $T \in
\MCscI$ will be entirely contained in an element of $\MTh$. 
For any element $K$, we still construct a box $\mB_K$ containing all
cells indexed by  $\forall (i_l, j_l) \in [i_{K, 1}, i_{K, 2}] \times
[j_{K, 1}, j_{K, 2}]$, 
where $i_{K, 1}(j_{K, 1})$, $i_{K, 2}(j_{K, 2})$ are indices
corresponding to the min/max
$x(y)$-coordinates of all
vertices as \eqref{eq_boxdomain}.
All cells that are entirely contained in $K$ are also included in
$\mB_K$. 
For every $T \in \mB_K$, $T \subset K$ is equivalent to all vertices
of $T$ are located in $K$, which can be readily checked. 
Then, any vertex of $K$ can be specified as $\varphi(T)$.
Moreover, since $T \subset K$, we also mark $K$ as the host element
for $T$, and the whole algorithm can be simplified and accelerated
because any point $\bm{p} \in T$ immediately gives $\bm{p} \in K$.
We note that although selecting a very fine background grid will
provide a slightly better computational efficiency because the set
$\MCscI$ will have more cells in this case, 
the initialization will meanwhile become very time-consuming as $s$
approaches zero, and also there is no need to construct a grid with
very small $s$ in our method.

We present the initializing step in Algorithm~\ref{alg_initialize2d}.
Generally speaking, the initialization has an $O(n_e)$ computational
complexity, where $n_e$ denotes the number of elements in $\MTh$. 
Though the computational time grows linearly as the mesh is
refined, the initialization only needs to be implemented once for a
given mesh.
In addition, the particle locating algorithm \eqref{eq_track2d} is
presented in Algorithm~\ref{alg_track2d}.

\begin{figure}[htp]
  \begin{minipage}[t]{0.46\textwidth}
    \footnotesize
    \begin{center}
      \begin{tikzpicture}[scale=1.8]
        \coordinate (A) at (0, -0.1);
        \coordinate (B) at (2.25, 1.08);
        \draw[thick] (A) -- (B);
        \foreach \y in {-1,...,5}
        \foreach \x in {-1,...,8} 
        \draw[thick, fill = black] (\x*0.38-0.13, -0.25 + \y*0.38)  circle [radius=0.02];
        \draw[thick, dashed] ($(A) + (-0.3, -0.3)$) rectangle ($(B) +
        (0.50, 0.3)$);
        \draw[thick, dashed, blue] (-0.13, -0.25) rectangle (0.25, 0.13);
        \foreach \y in {3}
        \foreach \x in {6} 
        \draw[thick, dashed, blue]  (\x*0.38-0.13, -0.25 + \y*0.38)
        rectangle  (\x*0.38 + 0.25, 0.13 + \y*0.38);
        \node[above] at (A) {$A$};
        \node[right] at (B) {$B$};
        \draw[thick, fill = red] (A) circle [radius=0.02];
        \draw[thick, fill = red] (B) circle [radius=0.02];
        \node at (1.8, -0.52) {\footnotesize box $\mB_e$};
        \node at (0.7, -0.1) {\footnotesize $(i_{e, 1}, j_{e, 1})$};
        \node at (2.35, 0.72) {\footnotesize $(i_{e, 2}, j_{e, 2})$};
      \end{tikzpicture}
    \end{center}
  \end{minipage}
  \begin{minipage}[t]{0.46\textwidth}
    \footnotesize
    \begin{center}
      \begin{tikzpicture}[scale=1.8]
        \coordinate (A) at (-0.8, -0.29);
        \coordinate (B) at (0.75, -0.02);
        \coordinate (C) at (0, 1.15);
        \coordinate (D) at (0.0, 0.39);
        \draw[thick] (A) -- (B) -- (C) -- (A);
        \node[left] at (A) {$A$};
        \node[below] at ($(A) + (0, -0.05)$) {$(i_{K, 1}, j_{K, 1})$};
        \node[right] at ($(C) + (0.6, 0.15)$) {$(i_{K, 2}, j_{K, 2})$};
        \node[right] at (B) {$B$};
        \node[above] at (C) {$C$};
        \draw[thick, dashed] ($(A) + (-0.03, -0.03)$) rectangle ($(C) +
        (0.8, 0.035)$);
        \node at (1.13, 0.39) {box $\mB_K$};
        \foreach \x in {0,...,9} 
        \draw[thick, fill = black] (\x*0.3 -1.35, -0.65)  circle [radius=0.02];
        \foreach \x in {0,...,9} 
        \draw[thick, fill = black] (\x*0.3 -1.35, -0.35)  circle [radius=0.02];
        \foreach \x in {0,1,2,6,7,8,9} 
        \draw[thick, fill = black] (\x*0.3 -1.35, -0.05)  circle [radius=0.02];
        \foreach \x in {3,4,5}
        \draw[thick, fill = blue, blue] (\x*0.3 -1.35, -0.05)  circle [radius=0.02];
        \foreach \x in {0,1,2,6,7,8,9} 
        \draw[thick, fill = black] (\x*0.3 -1.35, 0.25)  circle [radius=0.02];
        \foreach \x in {3,4,5}
        \draw[thick, fill = blue, blue] (\x*0.3 -1.35, 0.25)  circle [radius=0.02];
        \foreach \x in {0,1,2,3,6,7,8,9} 
        \draw[thick, fill = black] (\x*0.3 -1.35, 0.55)  circle [radius=0.02];
        \foreach \x in {4,5}
        \draw[thick, fill = blue, blue] (\x*0.3 -1.35, 0.55)  circle [radius=0.02];
        \foreach \x in {0,1,2,3,6,7,8,9} 
        \draw[thick, fill = black] (\x*0.3 -1.35, 0.85)  circle [radius=0.02];
        \foreach \x in {4,5}
        \draw[thick, fill = blue, blue] (\x*0.3 -1.35, 0.85)  circle [radius=0.02];
        \foreach \x in {0,1,2,3,4,5,6,7,8,9} 
        \draw[thick, fill = black] (\x*0.3 -1.35, 1.15)  circle [radius=0.02];
        \foreach \x in {0,...,9} 
        \draw[thick, fill = black] (\x*0.3 -1.35, 1.45)  circle [radius=0.02];
        \foreach \x in {3,4,5}
        \foreach \y in {0}
        \draw[thick, dashed, blue] (\x*0.3 -1.35, -0.05+\y*0.3) rectangle (\x*0.3 -
        1.05, 0.25+\y*0.3);
        \foreach \x in {4}
        \foreach \y in {1, 2}
        \draw[thick, dashed, blue] (\x*0.3 -1.35, -0.05+\y*0.3) rectangle (\x*0.3 -
        1.05, 0.25+\y*0.3);
      \end{tikzpicture}
    \end{center}
  \end{minipage}
  \caption{The box $\mB_K$, and all blue nodes are mapped with $K$.}
  \label{fig_box}
\end{figure}

\begin{algorithm}[htp]
  \caption{Initializing in two dimensions}
  \label{alg_initialize2d}
  \KwIn{the mesh $\MTh$, the Cartesian grid $\MCs$;}
  construction of $\varphi$: \\[1mm]
  \For{each $e \in \MEh$}{
  construct the box $\mB_e$ from $\mN_e = \{\bv_1, \bv_2\}$; \\[1mm]
  \For{each $T \in \mB_e$}{
  \If{there exists $t \in [0, 1]$ such that $\bm{q} = \bv_1 +
  t(\bv_2 - \bv_1) \in T$}{ 
  \If{$d(\bm{q}, \bv_1) < d(\bm{q}, \bv_2)$}{
  let $\varphi(T) = \bv_1$; 
  }\Else{
  let $\varphi(T) = \bv_2$;
  }}}}
  \For{each $K \in \MTh$}{
  construct the box $\mB_K$ from $\mN_K$; \\[0mm]
  \For{each $T \in \mB_K$}{
  \If{all vertices of $T$ located in $K$}{
  let $\varphi(T)$ be any vertex in $\mN_K$; \\[1mm]
  mark $K$ as the host element for $T$;
  }}}
  construction of $\{\theta_{\bv, i}\}_{i =1}^n$ for all vertices; \\[1mm]
  \For{each $\bv \in \mN$}{
  \For{each $e_{\bv, i} \in \mE_{\bv}$}{
  compute $\theta_{\bv, i}$ from the vertices of $e_{\bv, i}$ using
  \eqref{eq_pangle}; \\[1mm]
  }}
\end{algorithm}

\begin{algorithm}[htp]
  \caption{\revise{Tracking}{Locating} algorithm in two dimensions: }
  \label{alg_track2d}
  \KwIn{the point $\bm{p}$;}
  \KwOut{the host element $K$;}
  find the host cell $T \in \MCsc$ with the Cartesian grid; \\[1mm]
  \If{$T \in \MCscI$}{ return the host element $K$ of $T$; }
  \Else{
  set $\bv = \varphi(T)$; \\[1mm]
  compute $\theta$ from $\overline{\bv \bm{p}}$ using
  \eqref{eq_pangle}; \\[1mm]
  seek $K$ by searching $\theta$ in $\{\theta_{\bv,i}\}_{i = 1}^n$;
  \\[1mm] 
  return $K$;
  }
\end{algorithm}

In three dimensions, the preparation stage is similar to two
dimensions, where we will establish mappings $\varphi$ and $\phi$ and
construct the set $\MCNhssc$ simultaneously.
We first initialize $\MCNhssc$ as empty.
For any edge $e \in \MEh$ with $\mN_e = \{ \bv_1, \bv_2\}$, we
construct a box $\mB_e$ containing all cells indexed by $\forall (i_l,
j_l, k_l) \in [i_{e, 1} - 1, i_{e, 2} + 1] \times [j_{e, 1} - 1, j_{e,
2} + 1] \times
[k_{e, 1} - 1, k_{e, 2} + 1]$, where $i_{e, 1}(j_{e, 1}, k_{e, 1})$,
$i_{e,2}(j_{e, 2}, k_{e, 2})$ are indices from the min/max
$x(y,z)$-coordinates to all vertices as \eqref{eq_boxdomain}.
By the proof to Lemma \ref{le_patch3d}, we are required to find out
all cells whose circumscribed balls are cut by $\overline{e}$, and
such cells are included in $\mB_e$.
Let $p(t) = \bv_1 + t(\bv_2 - \bv_1)$ be the parameter equation to the
line $L_e$ along $e$. 
For any $T \in \mB_e$, if $d(\bom_T, L_e) < \frac{\sqrt{3}}{2} s$, we
compute $t_1 < t_2$ such that $p(t_1), p(t_2)$ are intersection points
between $B_T$ and $L_e$. 
Then, $\overline{B}_T$ and $\overline{e}$ intersecting is equivalent
to $[0, 1] \cap [t_1, t_2] \neq \varnothing$, which can be easily
checked. 
By this procedure, all cells having circumscribed balls cut by $e$ are
found. 
For such $T$, we let $t$ be anyone in $[0, 1] \cap [t_1, t_2]
$, and the closest vertex of $e$ to $p(t)$ 
can be selected as $\varphi(T)$ from the proof of Lemma
\ref{le_patch3d}.
By Lemma \ref{le_celledge3d}, for any $T \in \MCNhsc$ that intersects
with $\overline{e}$, any point $\bq$ in the intersection
satisfies that $d(\bq, \bv_i) >  \frac{w^*}{1 + \sin \alpha}$ for both
$i = 1, 2$.
Hence, if there holds $d(p(t) ,\bv_i) >  \frac{w^*}{1 + \sin \alpha}$
for both $i = 1, 2$, we know that $T \in \MCMEhc$ with $\phi(T) = e$,
and we add $T$ into the set $\MCNhssc$.

We next consider the cells that have intersection with some faces. 
For any $T \in \MCsc$, we construct a set $\mF_T$ consisting of all
faces intersecting $\overline{T}$. 
For any face $f \in \MFh$ with $N_f = \{\bv_1, \bv_2, \bv_3\}$, 
we let $i_{f, 1}(j_{f, 1}, k_{f, 1})$, $i_{f, 2}(j_{f, 2}, k_{f, 2})$ 
be indices corresponding to the min/max $x(y, z)$-coordinates from
vertices of $f$. 
We similarly define the box $\mB_f$ containing all cells indexed 
by $\forall (i_l, j_l, k_l) \in [i_{f, 1}, i_{f, 2}] \times [j_{f, 1},
j_{f,2}] \times [k_{f,1}, k_{f,2}]$. 
All cells cut by $\overline{f}$ are contained in $\mB_f$.
For each cell $T \in \mB_f$, we let $\bm{x}_{T, 1}, \bm{x}_{T, 2}$ be
two vertices of $T$ with $\bm{x}_{T, 1} < \bm{x}_{T, 2}$. 
Then, $\overline{T}$ intersecting $\overline{f}$ is equivalent to the
inequalities $\bm{x}_{T, 1} \leq \bv_3 + t_1(\bv_1 - \bv_3) +
t_2(\bv_2 - \bv_3) \leq \bm{x}_{T, 2}$, $t_1 + t_2 \leq 1$, $t_1 \geq
0$, $t_2 \geq 0$ admit a solution $(t_1, t_2)$, which can be easily
solved by methods of finding feasible solutions.
If such $(t_1, t_2)$ exists, we add $f$ into the set $\mF_T$.  All
sets $\mF_T(\forall T \in \MCsc)$ can be constructed in a piecewise
manner. 
Then,
for any $T \in \MCsc$, if $\mF_T$ is empty, we know that $T \in
\MCscI$, whose host element will be given later.
If $\mF_T = \{f \}$ only has one member, from the definition
\eqref{eq_cell3d} any vertex and any edge of $f$ can be selected as
$\varphi(T)$ and $\phi(T)$, respectively, and also we add $T$ to
$\MCNhssc$.
If $\mF_T$ has at least two elements and $\varphi(T)$, $\phi(T)$ have
not been determined, then $T \in \MCscB$ and there exist
two faces $f_1, f_2 \in \mF_T$ such that $f_1, f_2$ are faces of an
element sharing a common edge $e$. 
By the proof to Lemma \ref{le_patch3d}, the point $\bchi$ on the line
along $e$ can be readily computed. 
The closest vertex of $e$ to $\bchi$ can be chosen as $\varphi(T)$. 
If $d(\bchi, \bzta) > \frac{w^*}{1 + \sin \alpha}$ for $\forall \bzta
\in \mN_e$, we add $T$ to $\MCNhssc$ and give $\phi(T) = e$.

We finally consider cells in $\MCsc$. 
For any element $K$, we also construct a box $\mB_K$ containing cells
indexed by $\forall (i_l, j_l, k_l) \in [i_{K, 1}, i_{K, 2}] \times
[j_{K, 1}, j_{K, 2}] \times [k_{K, 1}, k_{K, 2}]$, where
$i_{K,1}(j_{K, 1}, k_{K, 1})$,  $i_{K,2}(j_{K, 2}, k_{K, 2})$ are
indices corresponding to min/max $x(y, z)$-coordinates to vertices of
$K$.
We know that all cells located in $K$ are also included in $\mB_K$.
For any $T \in \mB_K$, $T \subset K$ is equivalent to $\bv \in K$ for
all vertices $\bv \in \mN_K$.
If $T \subset K$, we also mark $K$ as the host element to $T$ since
$\bq \in T$ directly gives $\bq \in K$.

The initialization step is shown in Algorithm \ref{alg_initialize3d},
which also has the computational complexity of $O(n_e)$. 
The \revise{tracking}{particle locating} algorithm \eqref{eq_track3d}
is given in Algorithm \ref{alg_track3d}.

\begin{algorithm}[htp]
  \caption{Initializing in three dimensions}
  \label{alg_initialize3d}
  \KwIn{the mesh $\MTh$, the Cartesian grid $\MCs$;}
  construction of $\varphi$ and $\phi$ and $\MCNhssc$: \\[1mm]
  \For{each $e \in \MEh$}{ 
  construct the box $\mB_e$ from $\mN_e = \{\bv_1, \bv_2\}$; \\[1mm]
  \For{each $T \in \mB_e$}{
  \If{there exists $t \in [0, 1]$ such that $\bq = \bv_1 + t(\bv_2
  - \bv_1) \in \overline{B}_T$}{
  \If{$d(\bq, \bv_1) < d(\bq, \bv_2)$}{
  let $\varphi(T) = \bv_1$; 
  }\Else{
  let $\varphi(T) = \bv_2$;
  }
  \If{$d(\bq, \bv_1) > \frac{w^*}{1 + \sin \alpha}$ and $d(\bq, \bv_2)
  > \frac{w^*}{1 + \sin \alpha}$ }{
  add $T$ to $\MCNhssc$ with $\phi(T) = e$;
  }}}} 
  \For{each $f \in \MFh$}{
  construct the box $\mB_f$ from $\mN_f$; \\[1mm]
  \For{each $T \in \mB_f$}{
  \If{$T$ intersects $f$}{
  add $f$ to $\mF_T$;
  }}}
  \For{each $T \in \MCsc$}{
  \If{$\varphi(T)$ and $\phi(T)$ have been valued}{ continue;}
  \If{$\mF_T = \{ f \}$ has one member}{  
  let $\varphi(T)$ be any vertex of $f$ and let $\phi(T)$ be any edge
  of $f$; \\[1mm]
  add $T$ to $\MCNhssc$;
  }
  \If{$\#\mF_T \geq 2$}{
  find $f_1, f_2 \in \mF_T$ to be faces of an element sharing a common
  edge $e$ with $\mN_e = \{\bv_1, \bv_2\}$; \\[1mm]
  calculate the point $\bchi$; \\[1mm]
  \If{$d(\bchi, \bv_1) < d(\bchi, \bv_2)$}{
  let $\varphi(T) = \bv_1$; 
  }\Else{
  let $\varphi(T) = \bv_2$;
  }
  \If{$d(\bchi, \bv_1) > \frac{w^*}{1 + \sin \alpha}$ and $d(\bchi, \bv_2)
  > \frac{w^*}{1 + \sin \alpha}$ }{
  add $T$ to $\MCNhssc$ with $\phi(T) = e$;
  }}}
  \For{each $K \in \MTh$}{
  construct the box $\mB_K$ from $\mN_K$; \\[0mm]
  \For{each $T \in \mB_K$}{
  \If{all vertices of $T$ located in $K$}{
  add $T$ to $\MCNhssc$, and let $\varphi(T)$ and $\phi(T)$ be any
  vertex and any edge of $K$; \\[1mm]
  mark $K$ as the host element for $T$;
  }}}
  construction of $\{\theta_{e, i}\}_{i =1}^n$ for all edges; \\[1mm]
  \For{each $e \in \MEh$}{
  mark a vertex as $\bv_e$ and store the plane $P_e$; \\[1mm]
  \For{each $e_{j} \in \mE_{\bv_e}$}{
  compute $\theta_{e, j}$ from vectors $\bm{w}_{e, j}$ using
  \eqref{eq_pangle} on the plane $P_e$; \\[1mm]
  }}
\end{algorithm}

\begin{algorithm}[htp]
  \caption{\revise{Tracking}{Locating} algorithm in three dimensions}
  \label{alg_track3d}
  \KwIn{the point $\bm{p}$;}
  \KwOut{the host element $K$;}
  find the host cell $T \in \MCsc$ with the Cartesian grid; \\[1mm]
  \If{$T \in \MCscI$}{ return the host element $K$ of $T$; }
  \Else{
  \If{$\phi(T) = -1$} {
  set $\bv = \varphi(T)$ and compute $\wt{\bm{p}}$ using
  \eqref{eq_newp}; \\[1mm]
  find the host cell $\wt{T} \in \MCsc$ such that $\wt{\bm{p}} \in
  \wt{T}$;
  \\[1mm]
  update $T = \wt{T}$ and $\bm{p} = \wt{\bm{p}}$;}
  set $\bv = \varphi(T)$ and $e = \phi(T)$; \\[1mm]
  compute $\theta$ on the plane $P_e$ from $\overrightarrow{\bv_e
  \bm{p}}$; \\[1mm]
  seek $K$ by searching $\theta$ in $\{\theta_{e, j}\}_{j = 1}^n$;
  \\[1mm]
  return $K$;
  }
\end{algorithm}

\section{Numerical Results}
\label{sec_numericalresults}
In this section, we present a series of numerical experiments in two
and three dimensions to demonstrate the performance of the
proposed algorithm.
\revise{}{
All numerical tests are carried out on a computer equipped
with an Intel Core i7-13700K CPU and 64GB RAM.}

\subsection*{Example 1}
In the first example, we test our method in two dimensions. 
The computational domain is selected as $\Omega = (-1, 1)^2$ and the
background domain is chosen to be $\Omega^* = (-1 - \tau, 1 + \tau)^2$
with $\tau = 0.05$.
We consider a series of triangular meshes over $\Omega$ with the mesh
size $h = 1/10, 1/20, 1/40, 1/80$, see Fig.~\ref{fig_triangularmesh}.
The grid spacing $s$ is \justsubstitute{chosen}{taken} as $\frac{w^*
\sin \alpha}{\sqrt{2}(1 + \sin \alpha)}$ with $w^* = w - 0.005$.
In our setting, we first randomly generate $10^6$ points in the domain
$\Omega$, whose coordinates $(x, y) \sim (U(-1, 1))^2$ come from the
uniform distribution on $(-1, 1)$.
For each point $\bm{p}$ with its host element $K$, we generate 
random vectors $\bm{v} = (r \cos \theta, r \sin \theta)$ by $(r,
\theta) \sim U(0, \delta h_K) \times U(0, 2\pi)$ until
$\wt{\bm{p}} = \bm{p} + \bm{v} \in \Omega$.
Here $\delta$ is a parameter that measures the distance of the
trajectory in each movement for points.
We update their positions by letting $\bm{p} =
\wt{\bm{p}}$ for all points and seek their host elements again. 
In our tests, this locating process will
be repeated for 100 times. 

\revise{As a numerical comparison, we also adopt the neighbour searching
method \cite{Macpherson2009particle} to the host element for
every point in each step.}{As a numerical comparison, we adopt the
neighbour searching method \cite{Macpherson2009particle} and the
standard auxiliary structured grid method \cite{Seldner1988algorithms,
Li2019fast}
to locate points in each step.}
\justsubstitute{The main idea of this method}{The main idea of the
neighbour searching method} is to move the given
point to the next possible host element by finding the closest facet
of its host element that is intersected by the trajectory.
This searching process is repeated until the final host element is
found.
The implementation of this method is straightforward, but for the
point with a long trajectory, this algorithm might become inefficient.
\revise{}{
In the auxiliary structured grid method, for every structured cell 
a list of all unstructured elements that
intersect this cell is stored. 
After locating the given point in a structured cell, 
the host element is determined by a series of point-in-element tests
on the elements in the list of that cell.
Compared to the proposed method, this method can be regarded as
constructing a different element patch for every structured cell, 
whereas the point-in-element tests are required on element patches. }

We select $\delta = 0.1, 1, 5$ to test the methods for points with
short and long trajectories. 
The CPU times are reported in Tab.~\ref{tab_ex1}.
It can be seen that the computational time of the initialization stage
linearly depends on the number of elements.
We note that the initialization only needs to be done once before
performing the locating algorithm.
For different $\delta$, our algorithm has almost the same CPU times,
which illustrates that our method is robust to the position of the
point. 
For the neighbour searching method, the costed time increases
significantly for large $\delta$, and this method needs the trajectory
for the given particle while only coordinates are required in our
algorithm.
More importantly, even for small $\delta = 0.1$, our method is also
numerically detected to be more efficient than the neighbour searching
method. 
\revise{}{
The auxiliary structured grid method still needs a similar
initialization stage to prepare lists for structured cells, and here
we only report the times of locating particles. 
Similar to the proposed method, this method only uses the coordinates 
for locating, and the CPU times are independent of $\delta$. 
However, from numerical results, our method demonstrates a
significantly higher efficiency.
The reason is that the point-in-element tests are not required for
locating in the patch, which is a distinct advantage in our method.
}

\begin{figure}[htp]
  \centering
  \includegraphics[height=0.22\textwidth]{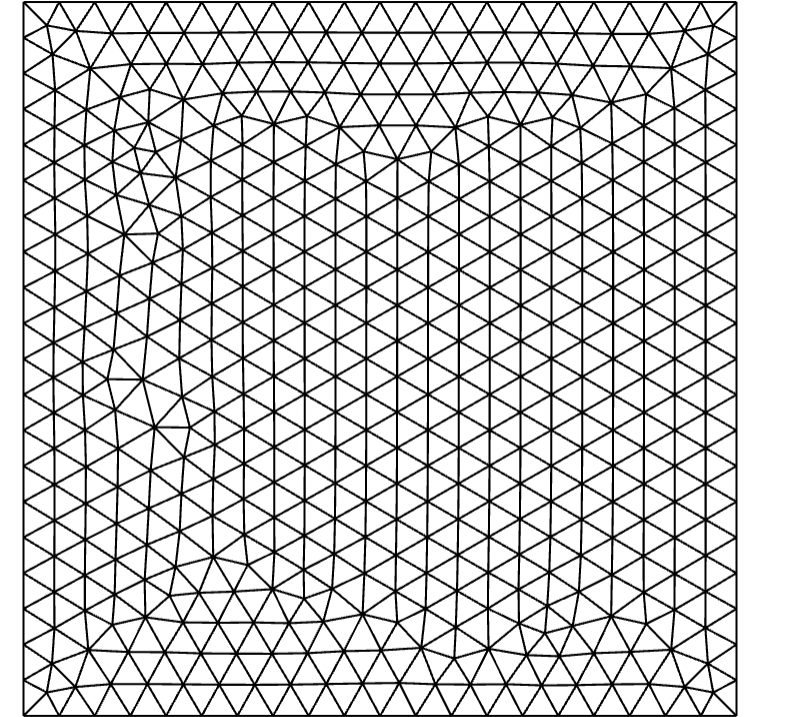}
  \hspace{50pt}
  \includegraphics[height=0.22\textwidth]{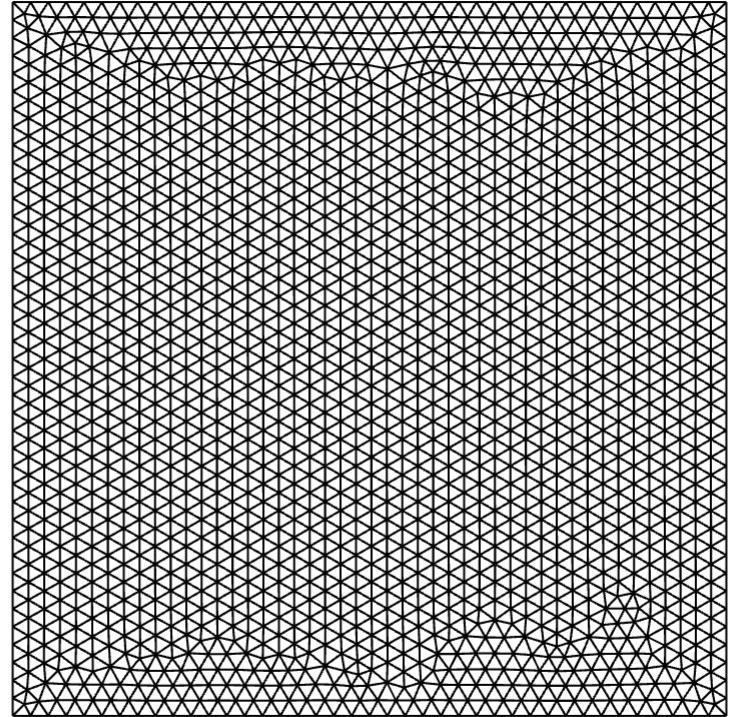}
  \caption{The triangular mesh $\MTh$ with $h = 1/10$ (left) / $h =
  1/20$ (right).}
  \label{fig_triangularmesh}
\end{figure}

\begin{table}
  \centering
  \renewcommand{\arraystretch}{1.5}
  \scalebox{0.75}{
  \begin{tabular}{p{1.5cm}|p{1.85cm}|p{1.35cm}
    |p{1.35cm}|p{1.35cm}|p{1.35cm}|p{1.35cm}|p{1.35cm}|p{1.35cm}|p{1.35cm}|p{1.35cm}}
    \hline\hline
    & \multicolumn{4}{c|}{patch searching method} 
    & \multicolumn{3}{c|}{neighbour searching method} 
    & \multicolumn{3}{c}{auxiliary structured grid method} \\
    \cline{2-11}
    & initialization & $\delta = 0.1$ & $\delta = 1$ & $\delta = 5$ & 
    $\delta = 0.1$ & $\delta = 1$ & $\delta = 5$ &
    $\delta = 0.1$ & $\delta = 1$ & $\delta = 5$ \\
    \hline
    $h = 1/10$ & 0.002
    & 1.369 & 1.436 & 1.426
    & 2.007 & 3.993 & 11.57 
    & 3.831 & 3.921 & 3.879 \\
    \hline
    $h = 1/20$ & 0.006 
    & 1.592 & 1.588 & 1.586
    & 2.259 & 4.575 & 13.99 
    & 4.773 & 4.758 & 4.701 \\
    \hline
    $h = 1/40$ & 0.023 
    & 1.820 & 1.816 & 1.812
    & 3.592 & 6.663 & 23.29 
    & 5.732 & 5.961 & 5.861 \\
    \hline
    $h = 1/80$ & 0.093 
    & 2.809 & 2.816 & 2.818
    & 4.945 & 10.28 & 43.28 
    & 8.552 & 8.718 & 8.816 \\
    \hline\hline
  \end{tabular}}
  \caption{The CPU times in Example 1.}
  \label{tab_ex1}
\end{table}

\subsection*{Example 2} 
In this example, the computational domain $\Omega$ and the background
domain $\Omega^*$ are the same as Example 1.
Here we perform the algorithm on a family of polygonal
meshes with elements $N = 2082, 8272, 33181, 132773$, which contain 
both triangular and quadrilateral elements, see
Fig.~\ref{fig_polymesh}.
The polygonal meshes are generated by the package GMSH
\cite{geuzaine2009gmsh}.
Although the theoretical analysis is established on triangular meshes,
the proposed algorithm also works for polygonal meshes,
merely requiring that the relation mapping $\varphi$ can be
constructed. 
In this test, $\varphi$ is still constructed by Algorithm
\ref{alg_initialize2d}.
The numerical setting is the same as Example 1 by randomly generating
$10^6$ points in $\Omega$, and then move points by random vectors in
each locating step. 
The CPU times are displayed in Tab.~\ref{tab_ex2}, which are similar
to the results on triangular meshes.
The detailed analysis for polygonal meshes is now considered as a
future work.

\begin{figure}[htp]
  \centering
  \includegraphics[height=0.25\textwidth]{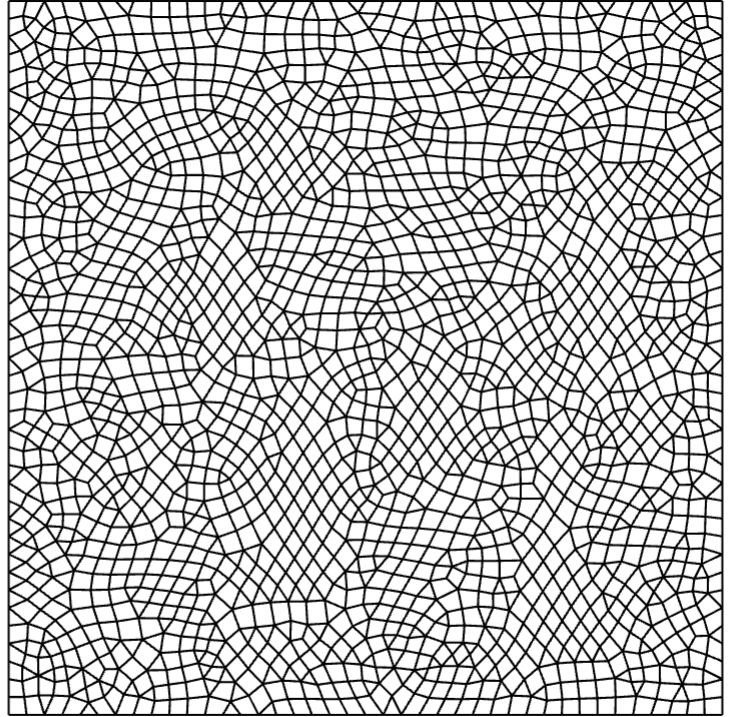}
  \hspace{30pt}
  \includegraphics[height=0.25\textwidth]{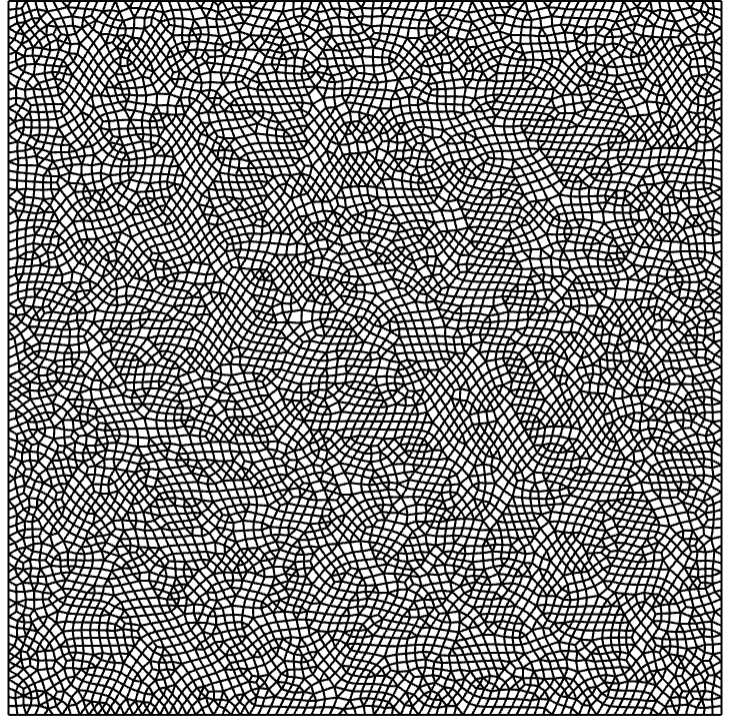}
  \caption{The polygonal mesh $\MTh$ with $2082$ elements (left) /
  $8272$ elements (right).}
  \label{fig_polymesh}
\end{figure}

\begin{table}
  \centering
  \renewcommand{\arraystretch}{1.5}
  \scalebox{0.8}{
  \begin{tabular}{p{1.5cm}|p{1.9cm}|p{1.25cm}
    |p{1.25cm}|p{1.25cm}|p{1.25cm}|p{1.25cm}|p{1.25cm}}
    \hline\hline
    \multirow{2}{*}{elements} & \multicolumn{4}{c|}{patch searching method} 
    & \multicolumn{3}{c}{neighbour searching method} \\
    \cline{2-8}
    & initialization & $\delta = 0.1$ & $\delta = 1$ & $\delta = 5$ & 
    $\delta = 0.1$ & $\delta = 1$ & $\delta = 5$ \\
    \hline
    $2082$ & 0.007 
    & 1.150 & 1.020 & 1.005 
    & 2.653 & 3.973 & 9.608 \\
    \hline
    $8272$ & 0.029
    & 1.136 & 1.152 & 1.125
    & 3.129 & 5.096 & 13.18 \\
    \hline
    $33181$ & 0.138
    & 1.938 & 1.939 & 1.923
    & 4.721 & 7.576 & 20.21 \\
    \hline
    $132773$ & 0.561
    & 2.729 & 2.743 & 2.711
    & 8.363 & 15.36 & 50.28 \\
    \hline\hline
  \end{tabular}}
  \caption{The CPU times in Example 2.}
  \label{tab_ex2}
\end{table}

\subsection*{Example 3}
This last example examines the proposed algorithm in three
dimensions. 
We choose the computational domain $\Omega = (0, 1)^3$ and select
$\Omega^* = (-\tau, 1 + \tau)^3$ covering the whole domain
$\Omega$ with $\tau = 0.05$.
The meshes are taken as a series of successively refined tetrahedral
meshes with the mesh size $h = 1/8, 1/16, 1/32, 1/64$, see
Fig.~\ref{fig_mesh3d}.
In the numerical setting, 
we also randomly generate $10^6$ points in $\Omega$ whose
coordinates $(x, y, z) \sim (U(0, 1))^3$.
In every step, the point $\bm{p}$ with the host element $K$ is moved
by a random vector $\bm{v} = (r \sin \theta_1 \cos \theta_2, r \sin
\theta_1 \sin \theta_2, r \cos \theta_1)$, where $(r, \theta_1,
\theta_2) \sim U(0, \delta h_K) \times U(0, \pi) \times U(0, 2 \pi)$. 
For every point, we update $\bm{p} = \bm{p} + \bm{v}$ and seek its new
host element in one step.
We repeat this particle locating process for $100$ times and record the CPU
times.

The results are collected in Tab.~\ref{tab_ex2}, and the proposed
method has a similar performance as two dimensions.
\revise{
For all $\delta$, the CPU times to our method are also almost the
same, and our method is numerically detected to be more efficient than
the neighbour searching method.}{
For the three-dimensional test, we also compare the proposed method
with the neighbour searching method and the auxiliary structured grid
method. 
For all $\delta$, our method is numerically demonstrated to be
significantly more efficient than the other methods.
}
In three dimensions, the CPU times costed by the preparation stage
also linearly depends on $n_e$, but here $n_e$ grows much faster than
two dimensions. 
Nevertheless, the initialization only needs to be implemented once for
a given mesh. 
In the case that there are a large number of points that require to be
located, our algorithm will have a remarkable advantage of efficiency.

\begin{figure}[htp]
  \centering
  \includegraphics[width=0.22\textwidth]{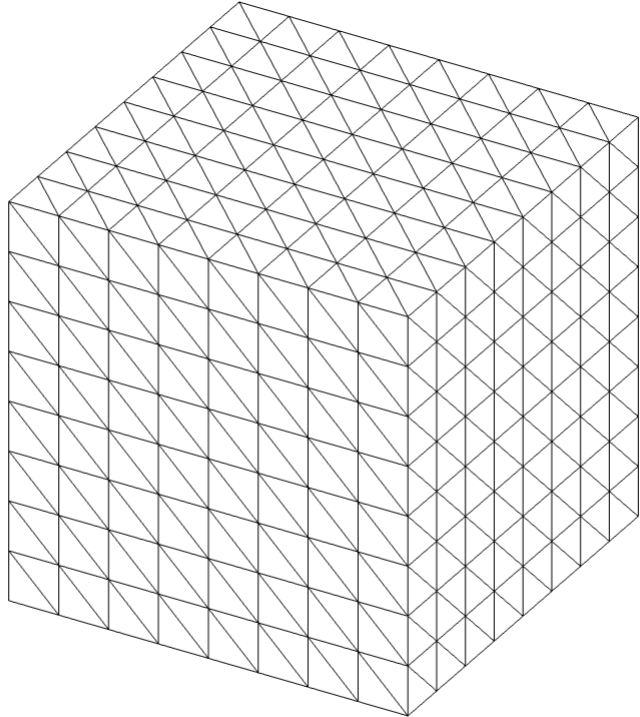}
  \hspace{60pt}
  \includegraphics[width=0.22\textwidth]{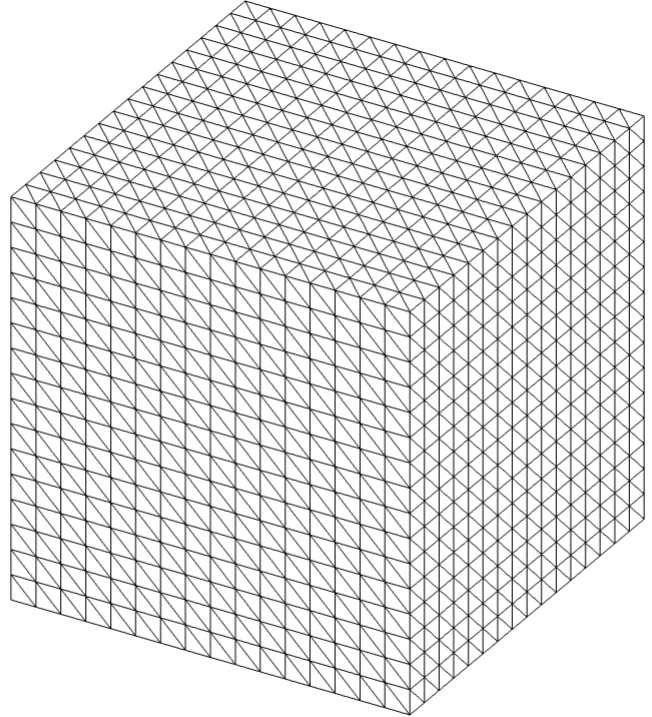}
  \caption{The tetrahedral meshes with $h = 1/8$ (left) / $h = 1/16$
  (right).}
  \label{fig_mesh3d}
\end{figure}

\begin{table}
  \centering
  \renewcommand{\arraystretch}{1.5}
  \scalebox{0.75}{
  \begin{tabular}{p{1.5cm}|p{1.9cm}|p{1.30cm}
    |p{1.30cm}|p{1.30cm}|p{1.30cm}|p{1.30cm}|p{1.30cm}
    |p{1.30cm}|p{1.30cm}|p{1.30cm}}
    \hline\hline
    & \multicolumn{4}{c|}{patch searching method} 
    & \multicolumn{3}{c|}{neighbour searching method} 
    & \multicolumn{3}{c}{auxiliary Cartesian grid method} \\
    \cline{2-11}
    & initialization & $\delta = 0.1$ & $\delta = 1$ & $\delta = 5$ & 
    $\delta = 0.1$ & $\delta = 1$ & $\delta = 5$ &
    $\delta = 0.1$ & $\delta = 1$ & $\delta = 5$ \\
    \hline
    $h = 1/8$ & 0.185
    & 4.165 & 4.163 & 4.171
    & 4.787 & 12.39 & 40.61 
    & 15.85 & 15.62 & 15.35 \\
    \hline
    $h = 1/16$ & 1.423 
    & 7.262 & 7.256 & 7.182
    & 8.296 & 22.83 & 69.43 
    & 25.29 & 25.76 & 25.66 \\
    \hline
    $h = 1/32$ & 11.57
    & 13.95 & 14.03 & 13.98
    & 16.32 & 53.51 &  196.08 
    & 42.69 & 43.71 &  42.61 \\
    \hline
    $h = 1/64$ & 89.51 
    & 18.35 & 18.27 & 18.19
    & 27.23 & 89.72 & 329.03  
    & 65.73 & 66.32 & 68.19  \\
    \hline\hline
  \end{tabular}}
  \caption{The CPU times in Example 3.}
  \label{tab_ex3}
\end{table}

\section{Conclusions}
\label{sec_conclusion}
\revise{}{
In this paper, we have introduced an efficient approach for particle 
locating on triangular and tetrahedral meshes in two and three dimensions. 
We first locate the given particle in a patch near a vertex, and then
seek the host element in the patch domain. 
The first step can be rapidly implemented by using an auxiliary
Cartesian grid with a prescribed grid spacing. 
In the second step, the task of finding the host element in the patch 
is shown to be equivalent to a searching problem, which can be
easily solved by standard searching algorithms.
The details of the computer implementation are presented in this paper.
Only coordinates of the given particles are required in the proposed
algorithms.
Numerical tests are carried out by locating randomly
distributed particles in both two and three dimensions.
The numerical results demonstrate remarkable advantages in terms of
efficiency of the proposed method.
}

\begin{appendix}
  \section{Geometrical Relations}
  \label{sec_app}
  In this appendix, we present some geometrical relations in two and
  three dimensions.
  \begin{lemma}
    For any edge $e \in \MEh$ with $\mN_e = \{\bv_1, \bv_2\}$, let
    $\bq \in e$ be a point on $e$ such that there exists a constant
    $\tau > 0$ satisfying $d(\bm{q}, \bv_1) > \tau$, $d(\bm{q},
    \bv_2) > \tau$, then there holds 
    \begin{equation}
      \begin{aligned}
        B(\bm{q}, \tau \sin \alpha) & \subset \mD{\mT_e}, &&
        \text{if} \ e \in \MEhi, \\
        (B(\bm{q}, \tau \sin \alpha) \cap \Omega) & \subset
        \mD{\mT_e}, && \text{if} \ e \in \MEhb,
      \end{aligned}
      \label{eq_app_BDT}
    \end{equation}
    where $\alpha$ is the minimum angle condition
    \eqref{eq_miniangle}.
    \label{le_app_disk}
  \end{lemma}
  \begin{proof}
    We mainly prove for the two-dimensional case $d = 2$.
    Let $K$ be any element in $\mT_e$ sharing the face $e$, and we let
    $\beta_1$ and $\beta_2$ be two base angles of $K$ on $e$
    corresponding to $\bv_1$ and $\bv_2$, respectively,
    see Fig.~\ref{fig_baseangles}. 
    From the condition \eqref{eq_miniangle}, we know that $\beta_1
    \geq \alpha$ and $\beta_2 \geq \alpha$.
    Let $e_1$ and $e_2$ be other two edges of $K$ that share a common
    vertex $\bv_1, \bv_2$ with $e$, respectively. 
    Let $B(\bm{q}, \tau \sin \alpha)$ be the disk centered at $\bm{q}$
    with the radius $\tau \sin \alpha$, and we define $B^+(\bm{q},
    \tau \sin \alpha)$ as the half disk of $B(\bm{q}, \tau \sin
    \alpha)$ formed by cutting $B(\bm{q}, \tau \sin \alpha)$ alone $e$
    with the same side of $K$.
    Let $L_1, L_2$ be lines along the edge $e_1, e_2$, respectively.
    Since $d_i = d(\bm{q}, \bv_i) > \tau$ for both $i = 1, 2$, we know
    that $d(\bm{q}, L_i) \geq d_i \sin \alpha > \tau \sin \alpha$, see
    Fig.~\ref{fig_baseangles}.
    The two estimates immediately bring us that the half disk
    $B^+(\bm{q}, \tau \sin \alpha) \subset K$ for $\forall K \in
    \mT_e$.
    Consequently, we conclude that $B(\bm{q}, \tau \sin \alpha)
    \subset \mT_{e}$ if $e \in \MEhi$ and $(B(\bm{q}, \tau \sin
    \alpha) \cap \Omega) \subset \mT_{e}$ if $e \in \MEhb$, i.e. the
    relation \eqref{eq_app_BDT} is reached.
    The proof can be directly extended to the case of $d = 3$ without
    any difficult.
    This completes the proof.

    \begin{figure}[htp]
      \centering
      \scriptsize
      \begin{tikzpicture}[scale = 3.9]
        \draw[thick] (0, 0) -- (1.2, 0);
        \draw[thick] (0, 0) -- (0.7, 0.9);
        \draw[thick] (0, 0) -- (1.2, 0) -- (0.7, 0.9);
        \node[left] at (0, 0) {$\bv_1$};
        \node[right] at (1.2, 0) {$\bv_2$};
        \draw[thick, fill=black] (0, 0) circle [radius=0.012];
        \draw[thick, fill=black] (1.2, 0) circle [radius=0.012];
        \draw[thick, dashed, ->] (0.7, 0) -- (1.06, 0.24);
        \node[] at (0.85, 0.2) {$h_{2}$};
        \draw[thick, dashed, ->] (0.7, 0) -- (0.31, 0.39);
        \node[] at (0.47, 0.35) {$h_1$};
        \fill (0.7,0) circle [radius=0.012] node[anchor=north] {$\bm{q}$};
        \draw (0.4,0) -- (0.4,0) node[midway, below] {$d_{1}$};
        \draw (1,0) -- (1,0) node[midway, below] {$d_{2}$};
        \draw (0.1,0) arc (0:53:0.1) node[midway, right] {$\beta_1$};
        \draw (1.1,0) arc (-180:-240:0.1);
        \draw[thick] (1.05,0) arc (0:180:0.35);
        \node at (0.75, 0.4) {$B^+(\bm{q}, \tau \sin \alpha)$};
        \node at (1.22, 0.07) {$\beta_2$};
        \node[left] at (0.28, 0.39) {$e_1$};
        \node[right] at (1.08, 0.27) {$e_2$};
      \end{tikzpicture}
      \caption{The element $K$ and the half disk $B^+(\bm{q}, \tau \sin
      \alpha) \subset K$.}
      \label{fig_baseangles}
    \end{figure}
  \end{proof}

  \newcommand{\tableincell}[2]{\begin{tabular}{@{}#1@{}}#2
  \end{tabular}}

  \section{list of Notation}

  \begin{table}[htp]
    \centering
    \renewcommand{\arraystretch}{1.53}
    \scalebox{0.75}{
    \begin{tabular}{p{3cm}p{6.2cm}p{8.3cm}}
      \hline\hline
      {\textbf{\itshape Notation}}
      & \textbf{Description} & \textbf{Definition/Note}
      \\
      \hline
      $\Omega$ & computational domain & - \\
      \hline
      $\MTh$ & mesh over $\Omega$ & $K$, $K_*$($K_1$, $K_2$, \ldots)
      usually denote
      elements in $\MTh$ \\
      \hline
      $h_K$, $\rho_K$, $w_K$ & parameters about the element $K$ & 
      \tableincell{l}{$h_K$: diameter of the circumscribed ball of
      $K$ \\
      $\rho_K$: radius of the inscribed ball of $K$ \\ 
      $w_K$: width of $K$}
      \\
      \hline
      $h$, $\rho$, $w$, $C_{\nu}$ & parameters about the mesh $\MTh$ & 
      \tableincell{l}{$h$: mesh size, $h := \max_{K \in \MTh} h_K$ \\
      $\rho := \min_{K \in \MTh} \rho_K$ \\
      $w := \min_{K \in \MTh} w_K$ \\
      $C_\nu$: regularity parameter, $h \leq C_{\nu} \rho$} 
      \\
      \hline
      $\alpha$ & minimum angle condition & definition
      \eqref{eq_miniangle} \\
      \hline
      $\MNh$, $\MNhi$, $\MNhb$ & set of nodes in $\MTh$ & 
      \tableincell{l}{$\MNh = \MNhi + \MNhb$ \\
      $\MNhi$: set of interior
      nodes \quad
      $\MNhb:$ set of boundary nodes \\ } \\
      \hline
      $\mN_K$ & set of vertices of $K$ & $\mN_K := \{ \bm{v} \in \MNh:
      \ \bm{v} \in \partial K \}$ \\
      \hline
      $B(\bm{z}, r)$ & ball centered at $\bm{z}$ with radius $r$ &
      $\partial B(\bm{z}, r)$: sphere of $B(\bm{z}, r)$
      \\
      \hline
      $\mT_{\bv}$, $\mD{\mT_{\bv}}$ & patch of the vertex $\bv$ & 
      \tableincell{l}{$\mT_{\bv}  := \{K \in
      \MTh: \ \bv \in \partial {K} \}$, set of elements around $\bv$ \\ 
      $\mD{\mT_{\bv}} := \text{Int}(\bigcup_{K \in \mT_{\bv}}
      \overline{K})$, domain of $\mT_{\bv}$ \\} \\
      \hline
      $\MEh$, $\MEhi$, $\MEhb$ & set of edges in $\MTh$ & 
      \tableincell{l}{$\MEh = \MEhi + \MEhb$ \\
      $\MEhi:$ set of interior
      edges \quad 
      $\MEhb:$ set of boundary edges \\} \\
      \hline
      $\mN_e$ & set of vertices of $e$ &  $\mN_e := \{ \bv \in \MNh: \
      \bv \in \overline{e}\} $ \\ 
      \hline
      $\mT_{e}$, $\mD{\mT_{e}}$ & patch of the edge $e$ &
      \tableincell{l}{$\mT_{e} := \{K \in \MTh: \ e \subset \partial K\}$:
      set of elements around $e$ \\ $\mD{\mT_e} := \text{Int}(\bigcup_{K \in
      \mT_e} \overline{K})$, domain of $\mT_{e}$ \\ }
      \\
      \hline
      $\mE_{\bv}$ & set of edges sharing the vertex $\bv$ &
      $\mE_{\bv} := \{e \in \MEh: \ \bv \in \mN_{e}\}$  \\
      \hline
      $\MFh$, $\MFhi$, $\MFhb$ & set of faces in $\MTh$ & 
      \tableincell{l}{$\MFh = \MFhi + \MFhb$ \\
      $\MFhi$: set of interior
      faces \quad
      $\MFhb$: set of boundary faces \\
      faces are used in three dimensions \\} \\
      \hline
      $\Omega^*$ & background domain covering $\Omega$ & $\Omega
      \subset \Omega^*$ \\
      \hline
      $\MCs$ & Cartesian grid on $\Omega^*$ & \tableincell{l}{$s$:
      grid spacing \\ $T$, $T_*$($T_1, T_2, \ldots$) usually denote
      cells in $\MCs$ \\} \\
      \hline
      $\MMs$ & set of nodes in $\MCs$ & - \\
      \hline
      $\mM_T$ & set of vertices of $T$ &  $\mM_T := \{ \bvs \in \MMs:
      \ \bvs \in \partial T \}$   \\
      \hline
      $\MCsc$ & set of all active cells &  $\MCsc := \{T \in \MCs: \ |T
      \cap \Omega| > 0\}$ \\ 
      \hline
      $d(\cdot, \cdot)$ & distance function & - \\
      \hline
      $\MCscI, \MCscB$, $\MCscF, \MCscE$ & active cells are divided
      into several types & \tableincell{l}{ definition
      \eqref{eq_ccell} in two dimensions \\
      definition \eqref{eq_cell3d} in three dimensions \\} \\
      \hline
      $\hat{\Theta}$ & pseudo angle function & definition
      \eqref{eq_pangle} \\
      \hline
      $\varphi$ & relation mapping from $\MCsc$ to $\MNh$ & cell $T$
      is included in the patch $\mD{\mT_{\varphi(T)}}$   \\
      \hline
      $\MCec$ & set of active Cartesian cells in $\mD{\mT_e}$ &
      $\MCec = \{T \in \MCsc: \ (T \cap \Omega) \subset \mD{\mT_e}
      \}$, $\MCMEhc := \bigcup_{e \in \MEh} \MCec$ 
      \\
      \hline
      $\MCBvc$ & set of cells intersecting with $\partial B(\bv, w^*)$
      &  $\MCBvc := \{T \in \MCsc: \ |\overline{T} \cap \partial
      B(\bv, w^*)| > 0 \}$ \\
      \hline
      $\MCNhssc$ & set of cells that are associated with edges, 
      $\MCNhsc \subset \MCNhssc \subset \MCMEhc$ &
      $\MCNhssc$ is constructed in the initializing stage
      simultaneously with mappings $\varphi, \phi$ \\
      \hline
      $\phi$ & mapping from $\MCNhssc$ to $\MEh$ & cell $T$ is
      included in the patch $\mD{\mT_{\phi(T)}}$ \\
      \hline\hline
    \end{tabular}
    }
    \caption{List of notation.}
    \label{tab_notation}
  \end{table}

\end{appendix}


\begin{thebibliography}{10}

\bibitem{Kuang2008new}
Kuang~S. B., Yu~A. B., and Zou Z.S., \emph{A new point-locating algorithm under
  three-dimensional hybrid meshes}, Int. J. of Multiphase Flow \textbf{34}
  (2008), 1023--1030.

\bibitem{Brandts2008equivalence}
J.~Brandts, S.~Korotov, and M.~K\v{r}\'{\i}\v{z}ek, \emph{On the equivalence of
  regularity criteria for triangular and tetrahedral finite element
  partitions}, Comput. Math. Appl. \textbf{55} (2008), no.~10, 2227--2233.

\bibitem{Capodaglio2017particle}
G.~Capodaglio and Aulisa E., \emph{A particle tracking algorithm for parallel
  finite element approximations}, Comput. \& Fluids \textbf{159} (2017),
  338--355.

\bibitem{Chen1999new}
X.~Q. Chen and Pereira J.~C. F., \emph{A new particle-locating method
  accounting for source distribution and particle-field interpolation for
  hybrid modeling of strongly coupled two-phase flows in arbitrary
  coordinates}, Numer. Heat Transfer, Part B \textbf{35} (1999), 41--63.

\bibitem{Chorda2002efficient}
R.~Chord\'a, J.~A. Blasco, and N.~Fueyo, \emph{An efficient particle-locating
  algorithm for application in arbitrary 2d and 3d grids}, Int. J. of
  Multiphase Flow \textbf{28} (2002), no.~9, 1565--1580.

\bibitem{geuzaine2009gmsh}
C.~Geuzaine and J.~F. Remacle, \emph{Gmsh: {A} 3-{D} finite element mesh
  generator with built-in pre- and post-processing facilities}, Internat. J.
  Numer. Methods Engrg. \textbf{79} (2009), no.~11, 1309--1331.

\bibitem{Haselbacher2007efficient}
A.~Haselbacher, F.~M. Najjar, and J.~P. Ferry, \emph{An efficient and robust
  particle-localization algorithm for unstructured grids}, J. Comput. Phys.
  \textbf{225} (2007), no.~2, 2198--2213.

\bibitem{Li2001efficient}
G.~Li and Modest~M. F., \emph{An efficient particle tracing schemes for
  structured/unstructured grids in hybrid finite volume/pdf monte carlo
  methods}, J. Comput. Phys. \textbf{173} (2001), 187--207.

\bibitem{Li2019fast}
Z.~Li, Y.~Wang, and L.~Wang, \emph{A fast particle-locating method for the
  arbitrary polyhedral mesh}, Algorithms (Basel) \textbf{12} (2019), no.~9,
  Paper No. 179, 16.

\bibitem{Lowier1990vectorized}
R.~L\"owier, \emph{A vectorized particle tracer for unstructured grids}, J.
  Comput. Phys. \textbf{87} (1990), no.~2, 496.

\bibitem{Lowier1995robust}
\bysame, \emph{Robust, vectorized search algorithms for interpolation on
  unstructured grids}, J. Comput. Phys. \textbf{118} (1995), 380--387.

\bibitem{Macpherson2009particle}
G.~B. Macpherson, Nordin N., and Weller~H. G., \emph{Particle tracking in
  unstructured, arbitrary polyhedral meshes for use in {CFD} and molecular
  dynamics}, Comm. Numer. Methods Engrg. \textbf{25} (2009), no.~3, 201--300.

\bibitem{Martin2009particle}
G.~D. Martin, E.~Loth, and Lankford D., \emph{Particle host cell determination
  in unstructured grids}, Comput. \& Fluids \textbf{38} (2009), 101--110.

\bibitem{Muradoglu2006auxiliary}
M.~Muradoglu and A.~D. Kayaalp, \emph{An auxiliary grid method for computations
  of multiphase flows in complex geometries}, J. Comput. Phys. \textbf{214}
  (2006), no.~2, 858--877.

\bibitem{Sani2009set}
M.~Sani and M.~S. Saidi, \emph{A set of particle locating algorithms not
  requiring face belonging to cell connectivity data}, J. Comput. Phys.
  \textbf{228} (2009), no.~19, 7357--7367.

\bibitem{Seldner1988algorithms}
D.~Seldner and T.~Westermann, \emph{Algorithms for interpolation and
  localization in irregular 2{D} meshes}, J. Comput. Phys. \textbf{79} (1988),
  no.~1, 1--11.

\bibitem{Wang2022GPU}
B.~Wang, I.~Wald, N.~Morrical, W.~Usher, L.~Mu, K.~Thompson, and R.~Hughes,
  \emph{An {GPU}-accelerated particle tracking method for
  {E}ulerian-{L}agrangian simulations using hardware ray tracing cores},
  Comput. Phys. Commun. \textbf{271} (2022), Paper No. 108221, 9.

\end{thebibliography}

\providecommand{\bysame}{\leavevmode\hbox to3em{\hrulefill}\thinspace}
\providecommand{\MR}{\relax\ifhmode\unskip\space\fi MR }
\providecommand{\MRhref}[2]{%
  \href{http://www.ams.org/mathscinet-getitem?mr=#1}{#2}
}
\providecommand{\href}[2]{#2}

\end{document}